\documentclass[12pt]{elsarticle}
\usepackage[a4paper, total={7in, 9in}]{geometry}
\usepackage{lineno,hyperref}
\usepackage{color}
\usepackage[table]{xcolor}
\usepackage{amsmath}
\usepackage{amssymb}
\usepackage{amsthm}
\usepackage{mathtools}
\usepackage{csquotes}
\usepackage{nicefrac}
\usepackage{tabularx}
\usepackage{csvsimple}
\usepackage{subcaption}
\usepackage[ruled,linesnumbered]{algorithm2e}
\usepackage{tikz}
\usepackage{pgfplots} 
\usepgfplotslibrary{statistics}
\pgfplotsset{compat=1.18}
\usetikzlibrary{arrows.meta,calc}
\usetikzlibrary{arrows.meta,positioning,calc,patterns,fit,backgrounds}
\usetikzlibrary{patterns}
\usetikzlibrary{decorations.pathreplacing}
\usepackage{longtable}
\usepackage{booktabs}
\usepackage{pdflscape}
\usepackage{graphicx}   
\usepackage{subcaption} 
\usetikzlibrary{arrows.meta, positioning}
\usepackage[T1]{fontenc}
\modulolinenumbers[5]
\RequirePackage{bm}
\makeatletter
\newcommand{\customlabel}[2]{%
\protected@write \@auxout {}{\string \newlabel {#1}{{#2}{\thepage}{#2}{#1}{}} }%
\hypertarget{#1}{#2}
}
\makeatother

\makeatletter
\def\ps@pprintTitle{%
\let\@oddhead\@empty
\let\@evenhead\@empty
\def\@oddfoot{}%
\let\@evenfoot\@oddfoot}
\makeatother

\journal{Discrete Optimization}

\biboptions{authoryear}

\newtheorem{theorem}{Theorem}

\newtheorem{proposition}[theorem]{Proposition}

\newtheorem{lemma}[theorem]{Lemma}

\newcommand{\NN}{\mathbb{N}}

\newcommand{\RR}{\mathbb{R}}

\newcommand{\Fres}{F^{\mathrm{res}}}
\newcommand{\Freb}{F^{\mathrm{reb}}}
\newcommand{\Floss}{F^{\mathrm{loss}}}
\newcommand{\Fmpd}{F^{\mathrm{MPD}}}

\newcommand{\Rloss}{R^{\mathrm{loss}}}
\newcommand{\Rmpd}{R^{\mathrm{MPD}}}

\newcommand{\rhorange}{\rho^{\mathrm{range}}}
\newcommand{\rhocorr}{\rho^{\mathrm{corr}}}
\newcommand{\rhocorract}{\rho^{\mathrm{corr}}}

\definecolor{prepExpTransit}{HTML}{1F77B4}
\definecolor{prepExpSupply}{HTML}{4A8BD6}
\definecolor{prepExpStorage}{HTML}{7FB0FF}
\definecolor{prepSupTransit}{HTML}{C40101}
\definecolor{prepSupSupply}{HTML}{F60000}
\definecolor{prepSupStorage}{HTML}{FC8080}
\definecolor{prepReinTransit}{HTML}{2CA02C}
\definecolor{prepReinSupply}{HTML}{4FBF4F}
\definecolor{prepReinStorage}{HTML}{8FDC8F}
\definecolor{prepRemaining}{HTML}{7F7F7F}
\definecolor{prepExpAgg}{HTML}{001CF3}
\definecolor{prepSupAgg}{HTML}{CD0000}
\definecolor{prepReinAgg}{HTML}{007600}

\newcommand{\swatch}[1]{%
  {\setlength{\fboxsep}{0pt}%
   \fcolorbox{black!35}{#1}{\rule{0pt}{1.4ex}\hspace{1.4ex}}}}

\newcommand{\Halmos}{$\Box$}

\begin{document}

\begin{frontmatter}
    
    \title{A General Multicriteria Optimization Perspective on Resilience}

    \author[kit]{Stephan Helfrich\corref{mycorrespondingauthor}}
    \ead{stephan.helfrich@kit.edu}
    \cortext[mycorrespondingauthor]{Corresponding author}
    
    \author[kit2]{Gabriela Ciolacu}
    \ead{gabriela.ciolacu@kit.edu}

    \author[ku]{Jan Boeckmann}
    \ead{jan.boeckmann@ku.de}
    
    \author[kit,kit2]{Emilia Grass}
    \ead{emilia.grass@kit.edu}
    
    \address[kit]{Karlsruhe Institute of Technology, Operations of Critical Infrastructures, Institute of Information Security and Dependability (KASTEL), Kaiserstraße~89, 76133 Karlsruhe}

    \address[kit2]{Karlsruhe Institute of Technology, Operations of Critical Infrastructures, Helmholtz Information \& Data Science School for Health (HIDSS4Health), Kaiserstraße~89, 76133 Karlsruhe}
    
    \address[ku]{Katholische Universität Eichstätt-Ingolstadt, Wirtschaftswissenschaftliche Fakultät Ingolstadt, Auf der Schanz~49, 85049 Ingolstadt}
    
    \begin{abstract}
        Resilience is a system's capability to prepare for, resist, absorb, and recover from adverse events.
        By definition, resilience therefore encompasses multiple criteria that can conflict and may prescribe different decisions. Importantly, a system must also maintain its effectiveness during routine operations and appropriately scale its preparedness for adverse events. Although resilience is inherently multicriteria, existing models often focus on a single criterion and do not explicitly analyze potential conflicts, the associated trade-offs, and their implications for decision-making.
        
        We formulate rebound, resistance, loss, and maximum performance degradation as separate resilience criteria in a general two-stage multicriteria model of network flows over time that integrates preparedness, effectiveness, and response. We show that, in general, such multicriteria optimization models are intractable and develop an enclosure-based heuristic for approximating the nondominated set. Our computational results show that different preparedness instruments and activation timings are associated with different resilience criteria and that the corresponding trade-offs are strongly instance-specific. Thus, resilience should be approached from an explicit multicriteria perspective rather than through a universal, preference-independent scalar index.
    \end{abstract}
    
\end{frontmatter}

\section{Introduction}\label{sec:intro}

As natural disasters and intentional attacks increase, enhancing resilience becomes increasingly important~\citep{linkov2019fundamental,IPCC2021}. Resilience is commonly understood as a system's capability to prepare, resist, absorb, and rebound from predictable and unpredictable adverse events, thereby preventing irreversible performance changes~\citep{national_research_council_disaster_2012}.

When designing optimization models to strengthen resilience, the system typically operates at a given performance level under routine operations until an \emph{adverse event} disrupts it. 
Hence, prior to the event, the network can be strengthened \emph{proactively} during routine operations, and resources can be prepared in advance. Once the adverse event starts, parts of the network become unavailable and system performance deteriorates. From this timestep onward, the system can only be adjusted \emph{reactively}, relying on resources established through proactive decisions to enhance resilience.
However, the precise formulation of resilience remains elusive. The formulation often lacks a unified understanding, leading to conceptual ambiguities about what constitutes resilience and how it differs from other related concepts (e.g., preparedness)~\citep{ciolacu2025resilience}. 
Additionally, this problem is particularly challenging as resilience is inherently multifaceted and cannot be captured by a single criterion. Accordingly, this work considers four commonly used resilience criteria, illustrated in Figure~\ref{fig:resilience-curve} and detailed below:

\begin{itemize}
    \item \emph{Rebound: }Upon the start of an adverse event, the system should return to its routine performance quickly.
    \item \emph{Resistance: }Upon the start of an adverse event, the system should keep operating at the routine performance as long as possible.
    \item \emph{Loss: }Upon the start of an adverse event, the cumulative loss of performance over time should be as small as possible.
    \item \emph{Maximum Performance Degradation (MPD): }Upon the start of an adverse event, the difference between the routine performance and the worst performance should be as small as possible.
\end{itemize}
Whereas the literature widely explores individual resilience criteria~\citep{sharkey_search_2020}, extant optimization models and frameworks typically capture only a narrow view of resilience and frequently overlook temporal criteria such as rebound and resistance~\citep{Gao2019}. Consequently, it remains unclear whether the four criteria conflict, how pronounced the associated trade-offs are, and how these conflicts affect preparedness and response decisions. In particular, preparedness resources must be allocated cost-effectively without unduly compromising performance during routine operations, but the implications of different preparedness decisions for the individual resilience criteria remain insufficiently understood. Against this background, we address the following research question: \emph{How can resilience be approached from a multicriteria perspective, and what do conflicts among resilience criteria imply for preparedness and response decisions?}

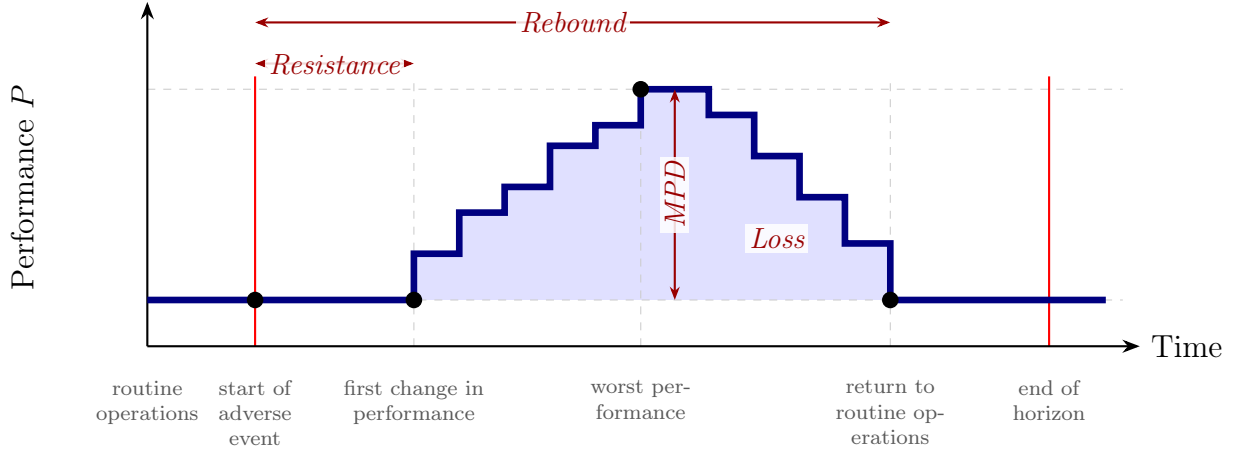
\begin{figure}
\centering
\begin{tikzpicture}[
    >=Stealth,
    x=1.5cm,
    y=3.4cm,
    axis/.style={thick},
    grid/.style={gray!40,dashed},
    curve/.style={blue!50!black,line width=2.5pt},
    crit/.style={red!60!black,thick,{Stealth[length=5pt]}-{Stealth[length=5pt]}},
    critlab/.style={red!60!black,font=\small\itshape,fill=white,fill opacity=0.85,
                    text opacity=1,inner sep=1.5pt},
    tick/.style={below,font=\scriptsize,align=center,text=black!65},
    dot/.style={circle,fill=black,inner sep=2.2pt},
    scale = 1
]

\def\tzero{-0.35}
\def\te{0.6}
\def\tres{2.0}
\def\td{4.0}
\def\tr{6.2}
\def\tf{7.6}
\def\tend{8.1}

\def\yBase{0.18}
\def\yPeak{1.00}

\fill[blue!12]
  plot[const plot] coordinates {
    (\tzero,\yBase) (\te,\yBase) (\tres,0.36) (2.4,0.52) (2.8,0.62)
    (3.2,0.78) (3.6,0.86) (\td,\yPeak) (4.6,0.90) (5.0,0.74) (5.4,0.58)
    (5.8,0.40) (\tr,\yBase) (\tf,\yBase) (\tend,\yBase)
  } -- (\tend,\yBase) -- (\tzero,\yBase) -- cycle;

\draw[grid] (\tzero,\yPeak) -- (\tend+0.15,\yPeak);
\draw[grid] (\tzero,\yBase) -- (\tend+0.15,\yBase);

\foreach \x in {\tres,\td,\tr}{\draw[grid] (\x,0) -- (\x,1.05);}
\draw[red,thick] (\te,0) -- (\te,1.05);
\draw[red,thick] (\tf,0) -- (\tf,1.05);

\draw[axis,->] (\tzero,0) -- (\tend+0.3,0) node[right] {Time};
\draw[axis,->] (\tzero,0) -- (\tzero,1.34);
\node[rotate=90] at (-1.45,0.62) {Performance $P$};

\node[below,font=\small] at (\tzero,0) {};
\node[below,font=\small] at (\te,0) {};
\node[below,font=\small] at (\tres,0) {};
\node[below,font=\small] at (\td,0) {};
\node[below,font=\small] at (\tr,0) {};
\node[below,font=\small] at (\tf,0) {};

\node[tick,text width=1.6cm] at (\tzero,-0.09) {routine \\operations};

\node[tick,text width=1.6cm] at (\te,-0.09) {start of \\ adverse event};
\node[tick,text width=2.3cm] at (\tres,-0.09) {first change in\\performance};
\node[tick,text width=2.3cm] at (\td,-0.09) {worst performance};
\node[tick,text width=2.3cm] at (\tr,-0.09) {return to \\routine operations};
\node[tick,text width=1.6cm] at (\tf,-0.09) {end of\\horizon};

\draw[curve]
  plot[const plot] coordinates {
    (\tzero,\yBase) (\te,\yBase) (\tres,0.36) (2.4,0.52) (2.8,0.62)
    (3.2,0.78) (3.6,0.86) (\td,\yPeak) (4.6,0.90) (5.0,0.74) (5.4,0.58)
    (5.8,0.40) (\tr,\yBase) (\tf,\yBase) (\tend,\yBase)
  };

\node[dot] at (\te,\yBase) {};
\node[dot] at (\tres,\yBase) {};
\node[dot] at (\td,\yPeak) {};
\node[dot] at (\tr,\yBase) {};


\draw[crit] (\te,1.10) -- (\tres,1.10);
\node[critlab] at ({(\te+\tres)/2},1.10) {Resistance};

\draw[crit] (\te,1.26) -- (\tr,1.26);
\node[critlab] at ({(\te+\tr)/2},1.26) {Rebound};

\draw[crit] (4.3,\yBase) -- (4.3,\yPeak);
\node[critlab,rotate=90] at (4.3,0.59) {MPD};

\node[critlab] at (5.2,0.42) {Loss};

\end{tikzpicture}

    \caption{General performance curve and illustration of the four resilience criteria in the context of an adverse event $E$. Performance is measured as unmet demand and is therefore minimized. Since we assume discrete time step, unmet demand is a step function.}
    \label{fig:resilience-curve}
\end{figure}

To address this, we build on the literature~\citep[e.g.,][]{alderson_assessing_2014} and propose a general two-stage optimization model that enhances resilience from a \emph{multicriteria perspective}. The model (i) accounts for all four resilience criteria and thereby allows potential conflicts and the associated trade-offs to be analyzed, (ii) follows a flow-over-time formulation, and (iii) incorporates proactive and reactive decisions. Including all resilience criteria enables us to assess whether the criteria conflict and, where they do, to quantify the associated trade-offs and their implications for decision-making. In addition, we propose two complementary conditions to balance a system's preparedness and effectiveness before the adverse event with its reaction to it.
Flow-over-time formulation \citep{Skutella2009} extends classical static flow models by explicitly accounting for temporal variation in flows, yielding a more realistic representation of system performance. The formulation can be adapted to a wide variety of systems, such as critical infrastructures~\citep[e.g.,][]{Bertsimas2013,alderson_assessing_2014,sharkey_interdependent_2015,ouyang_mathematical_2017}, and allows modeling of temporal aspects of resilience. Finally, the two-stage structure models proactive and reactive decisions, hence mapping the implications of preparedness and effectiveness on resilience criteria. 

To analyze the resulting multicriteria problem, we exploit the structural properties of the temporal resilience criteria. This structure enables to prove intractability of the multicriteria problem and motivates the development of a heuristic. We further analyze the implications of the resilience criteria and complementary conditions through numerical experiments on benchmark instances. The computational results demonstrate that the resilience criteria can induce materially different preparation and response decisions. On the  instances studied and across the computed sets of solutions, rebound varies more substantially than resistance, loss, and MPD, giving recovery-time targets the strongest discriminatory power among candidate solutions. The results provide evidence of conflicts among the resilience criteria, with the most consequential conflicts concerning temporal resilience. Rebound displays the strongest discriminatory power among candidate solutions, while the conflict between resistance and rebound also depends on when reserve capacities are activated. Furthermore, the magnitude of the associated trade-offs varies substantially across instances. Consequently, there is no universally preferred resilience strategy, and preferences among resilience criteria cannot generally be specified independently of the instance at hand. Temporal modelling is therefore necessary to capture the conflicts that matter most.

These findings also have an important implication for resilience-oriented operations research. A single resilience score may appear objective, but its construction necessarily encodes preferences regarding the underlying criteria. Depending on these preferences, the model can prescribe different preparedness investments and response strategies. Resilience should therefore be approached from an explicit multicriteria perspective rather than represented by a universal scalar index.

The contributions of this work are threefold. First, as a \emph{modelling contribution}, we formulate rebound, resistance, loss, and MPD as separate criteria in a general two-stage model of network flows over time that integrates preparedness, effectiveness, and response decisions. The formulation is deliberately general and can be adapted to a broad range of application contexts. Second, as a \emph{mathematical contribution}, we decompose the resulting problem into a family of bicriteria linear programs and show that the number of efficient temporal resilience profiles can grow superpolynomially in the encoding length, precluding a polynomial-size enumeration of the complete nondominated set in general. Motivated by this, we develop an enclosure-based heuristic for approximating the nondominated set. Third, we derive \emph{managerial insights} from a computational study on benchmark-derived instances. The results indicate that different preparedness instruments and activation timings are associated with different resilience criteria, that temporal resilience is central to the most consequential conflicts, and that the magnitude of the associated trade-offs is strongly instance-specific. Consequently, a universal, preference-independent scalar resilience index is unsuitable for prescribing decisions across different systems.

The outline of this work is as follows. In Section~\ref{sec:literature}, we introduce resilience and the background work. In Section~\ref{sec:problem-formulation}, we introduce the flows over time and provide general definitions of adverse events and adjustments in reaction to these adverse events. Additionally, we introduce the criteria that characterize a system's resilience under adverse events and show how to capture them in the model based on flows over time. In Section~\ref{sec:multiobjective}, we develop the general multicriteria model formulation and the heuristic method in Section~\ref{sec:solution}, followed by evaluation on benchmark instances in Section~\ref{sec:experiments}. The conclusion is presented in Section~\ref{sec:conclusion}.

\section{Related Works}\label{sec:literature}

In this section, we provide a brief overview of various ways to quantify resilience and optimization models for enhancing a system's resilience. We start by introducing the resilience criteria studied in the literature in  Section \ref{sec:literature_resilience}. For a broader understanding of resilience quantification and its interpretations, the reader is directed to \cite{liu2022network,bruckler2024review}.
In Section \ref{sec:literature_optimization}, we survey models that address multiple resilience criteria and include both proactive and reactive decisions.
For more general surveys on the application of optimization models to resilience, 
we refer to~\cite{ouyang_review_2014,gras_prepositioning_2015, ouyang_mathematical_2017,sharkey_search_2020,oveysi_optimization_2021}.

\subsection{Resilience Criteria and Complementary Conditions}\label{sec:literature_resilience}

Resilience describes the system's capability to prepare for, resist, absorb, and react to adverse events that disrupt routine operations. Resilience can be quantified with respect to two aspects: \emph{performance resilience} and \emph{temporal resilience}~\citep{liu2022network}. 

The performance formulation identifies the most critical managerial goal that must be consistently met for the system to be considered functional during both routine and adverse event operations~\citep{poulin2021infrastructure}. Hence, to determine whether the system has experienced significant losses, we compare its performance from the onset of the event with its performance under routine operations; the resulting difference is the \emph{performance gap}. 
For instance, met or unmet demand is often used as a performance metric when evaluating resilience~\citep[e.g.,][]{ni2018modeling,Gao2019}. But they yield different insights into adverse event operations. Unmet demand measures failure to maintain service continuity upon the start of an adverse event, emphasizing shortages, whereas met demand assesses the quantity of services or goods successfully provided.
Similarly, the duration of the period between the start of the adverse event and the return to routine operations also matters for resilience \citep{bruckler2024review}. For instance, a quick recovery limits performance losses and thus contributes positively to resilience.
Building on the existing literature~\citep{ciolacu2025resilience}, a set of resilience criteria for temporal and performance resilience is presented to capture the system's most essential capabilities in the context of adverse events. 

Temporal resilience can be assessed using \emph{Rebound Criterion} and \emph{Resistance Criterion}. \emph{Rebound} indicates the time it takes for a system to return to routine operation performance \citep{ivanov2017supply,Gao2019}. If the system fails to rebound and minimize the performance gap, its resilience is considered low.
Hence, a shorter phase contributes positively to system resilience.\footnote{In supply chain literature, the rebound criterion evaluates the capability of the system to recover \citep[e.g.,][]{bruckler2024review}, the system's time-to-recovery \citep[e.g.,][]{behzadi2020metrics}, and the level of recovery in comparison to a threshold \citep[e.g.,][]{behzadi2020metrics}.}
In contrast, \emph{Resistance} indicates the extent in time to which the system copes with the immediate impact of an adverse event and manages to prevent significant performance degradation~\citep[e.g.,][]{ivanov2017supply}. Resistance measures the phase during which the performance gap is minimal upon the start of an adverse event. Hence, a longer phase contributes positively to system resilience.\footnote{In supply chain literature, the resistance criterion evaluates the system's time-to-survive \citep{ivanov2017supply,bruckler2024review}.}

Performance resilience can be assessed using \emph{Loss Criterion} and \emph{Maximum Performance Degradation (MPD) Criterion}. \emph{Loss} is a well-known performance resilience criterion that evaluates the cumulative performance deficit ~\citep{bruckler2024review}. This criterion indicates whether the system can absorb and adapt to adverse events without irreversible performance degradation after an adverse event begins, compared with routine operations. Hence, a decrease in the system's total performance deficit over time contributes positively to resilience. 
In contrast, \emph{Maximum Performance Degradation} aims to prevent a large drop in performance upon the start of an adverse event~\citep {bruckler2024review}. Hence, a smaller performance gap upon the start of an adverse event contributes positively to resilience. 

Beyond the four resilience criteria, a system should remain effective in pursuing its goals during routine operations and cost-efficient in its preparedness decisions~\citep{dai2026two}. Hence, \emph{Preparedness} and \emph{Effectiveness} are introduced as \emph{complementary conditions} that balance the efforts to enhance resilience with the system's goals under routine operations. The scale of preparedness decisions taken to increase resilience should contribute to system resilience while remaining cost-effective~\citep{lucker_balancing_2025}. Additionally, greater resilience should not reduce performance under routine operations~\citep{lucker_balancing_2025}. 

\subsection{Enhancing Resilience Using Optimization}\label{sec:literature_optimization}
In the optimization literature, resilience is often approached through optimization frameworks that capture both proactive and reactive decisions \citep{sharkey_search_2020}. 
One of the first optimization models that address these decisions in a general setting was introduced in~\cite{brown_defending_2006}. The authors propose a static network flow formulation that maximizes flow while minimizing unused supply and unmet demand. The model assumes an intelligent attacker and implicitly treats the resulting performance loss as a proxy for resilience. 

Building on this work, \cite{alderson_assessing_2014} 
propose a general mathematical framework to improve the resilience of critical infrastructures under adverse events. The authors
model a wide range of critical infrastructures using shortest-path, network-flow, scheduling, or general (non-) linear (mixed-)integer programs. \cite{alderson_assessing_2014} implicitly treats the resulting performance loss as a proxy for resilience. 
Subsequent studies account for interdependencies between critical infrastructures~\citep{ouyang_mathematical_2017,belle_resilience-based_2023} and have been used to identify proactive and reactive decisions across various application domains, such as electric power grids~\citep{alguacil_trilevel_2014}, water distribution systems~\citep{wu_defenderattackeroperator_2021}, transportation and supply chains~\citep{miller-hooks_measuring_2012,ni2018modeling}, humanitarian relief chains ~\citep{taghvaei2025bi} and healthcare systems~\citep{goodarzian2022sustainable,ozcelik2025bi}.
In the context of classical maximum network flow problems, \cite{Hien2020} and \cite{Eshghali2023} assess the network's ability to fortify and restore, respectively, optimal flow following adverse events while considering preparedness. 
Similarly,
\cite{chen_where_2026} study the maximization of expected maximum flows while considering preparedness but interpret resilience as the criticality of vertices and arcs.

Network recovery is modeled through restorative optimization frameworks, such as ~\cite{sharkey_interdependent_2015,gonzalez_interdependent_2016,almoghathawi_exploring_2021,alkhaleel2022risk}. These frameworks combine a flow over time formulation with a scheduling problem. The frameworks identify an optimal sequence of restoration steps, thereby assuming that faster recovery corresponds to greater resilience. Notably, \cite{ghorbani-renani_protection-interdiction-restoration_2020} extend these frameworks by incorporating proactive decisions and balancing the MPD criterion with recovery. 

One of the first optimization models to explicitly account for temporal resilience was proposed by~\cite{goldbeck_resilience_2019}, who developed a flow over time and simulation framework to determine optimal reactions to adverse events for critical infrastructures. 
However, resilience is only evaluated, though not explicitly optimized, in terms of rebound, loss, and MPD criteria. Similarly, \cite{fang_adaptive_2019}, \cite{Gao2019} and \cite{ghorbani-renani_protection-interdiction-restoration_2020} indirectly evaluate the rebound criteria for (interdependent) critical infrastructures and supply chains.
Recently, \cite{helfrich_defender-attacker-defender_2026} propose a flow-based formulation for hospital service planning following a cybersecurity attack. In particular, the objective function explicitly incorporates rebound time, loss, and MPD through a weighted sum.

In summary, the literature enhances resilience either through system performance, primarily using loss and MPD criteria, or through restoration decisions in reaction to adverse events. Nonetheless, temporal criteria remain less represented. Although recent studies seek to integrate both performance and temporal resilience, existing frameworks remain application-specific rather than providing a general formulation of resilience.
Lastly, two-stage optimization \citep{alderson_assessing_2014,gras_prepositioning_2015} has emerged in the literature as a suitable framework for modeling resilience as a multicriteria problem. 
It has therefore been widely applied to resilience enhancement problems, particularly in humanitarian contexts, where resilience must often be discussed and balanced against competing objectives.

\section{Problem Setting and Modeling Framework}\label{sec:problem-formulation}
In this section, we formalize the problem setting underlying the two-stage multicriteria formulation for assessing and improving a system's resilience. We first provide a high-level description of the two-stage decision framework (Section~\ref{sec:decision-framework}). We then recall the definition of flows over time to capture the system's performance under routine operations (Section~\ref{sec:flow-over-time}). Building on this, we formally define adverse events and reactive flows over time that capture the system's operation in reaction to them (Section~\ref{sec:adverse-events-reactive-flows}). 
 
\subsection{Two-Stage Decision Framework}\label{sec:decision-framework}
Improving a system's resilience is driven by proactive decisions before some adverse event occurs combined with reactive decisions after the onset of adverse event. This separation gives rise to a two-stage decision process, whose sequencing across the planning horizon is illustrated in Figure~\ref{fig:temp-sequence}.

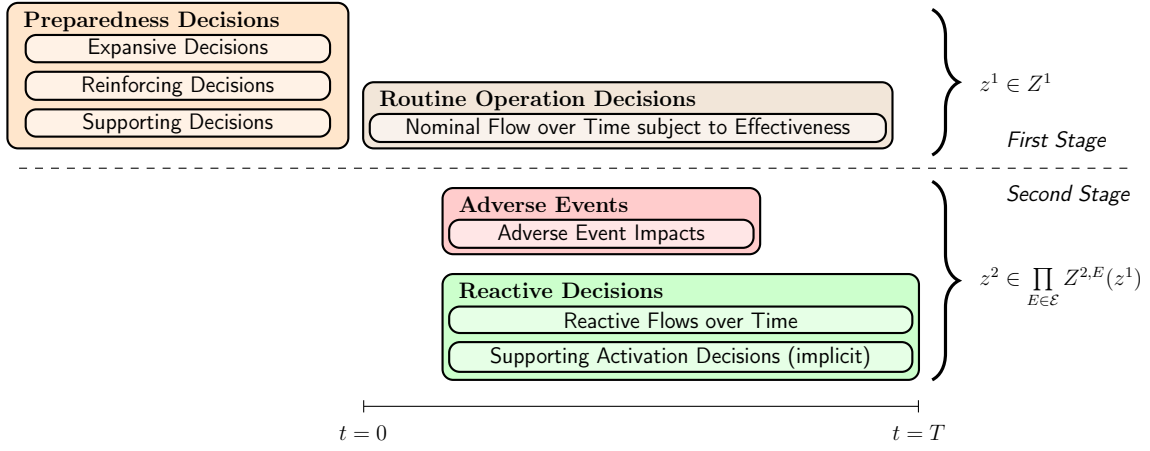
\begin{figure}
    \centering
    \begin{tikzpicture}[font=\sffamily, 
            stage/.style={rectangle, draw, thick, rounded corners, inner sep=2pt},
            label/.style={font=\bfseries},
            >=stealth,
            scale=0.7,transform shape
            ]
    
    \node[stage, fill=orange!20, minimum width=6.3cm, minimum height=2.75cm, text width = 6.3cm] (stage1) at (-5,6.75) {};
    \node[anchor=north west,label] at ([xshift=5pt,yshift=-1pt]stage1.north west) {Preparedness Decisions};
    \node[stage, fill=orange!10, minimum width=5.75cm, minimum height=.5cm] (coop)   at ([xshift=0pt,yshift=-25pt]stage1.north)  {Expansive Decisions};
    \node[stage, fill=orange!10, minimum width=5.75cm, minimum height=.5cm] (backup) at ([xshift=0pt,yshift=-45pt]stage1.north)  {Reinforcing Decisions};
    \node[stage, fill=orange!10, minimum width=5.75cm, minimum height=.5cm] (it)     at ([xshift=0pt,yshift=-65pt]stage1.north) {Supporting Decisions};
    
    \node[stage, fill=brown!20, minimum width=10cm, minimum height=1.25cm, text width = 3cm] (nominal) at (3.5,6) {};
    \node[anchor=north west,label] at ([xshift=5pt,yshift=-1pt]nominal.north west) {Routine Operation Decisions};
    \node[stage, fill=brown!10, minimum width=9.75cm, minimum height=.5cm] (nominalflow) at ([xshift=0pt,yshift=-25pt]nominal.north) {Nominal Flow over Time subject to Effectiveness};
    
    \draw[dashed] (-8,5) -- (13,5);
    
    \node[stage, fill=red!20, minimum width=6cm, minimum height=1.25cm] (stage2) at (3,4) {};
    \node[anchor=north west, label] at ([xshift=5pt,yshift=-1pt]stage2.north west) {Adverse Events};
    \node[stage, fill=red!10, minimum width=5.75cm, minimum height=.5cm] (attack) at ([xshift=0pt,yshift=-25pt]stage2.north) {Adverse Event Impacts};
    
    \node[stage, fill=green!20, minimum width=9cm, minimum height=2cm] (stage3) at (4.5,2) {};
    \node[anchor=north west, label] at ([xshift=5pt,yshift=-1pt]stage3.north west) {Reactive Decisions};
    \node[stage, fill=green!10, minimum width=8.75cm, align=center, minimum height=.5cm] (backup3) at ([xshift=0pt,yshift=-25pt]stage3.north) {Reactive Flows over Time};
    \node[stage, fill=green!10, minimum width=8.75cm, minimum height=.5cm] (coop3) at ([xshift=0pt,yshift=-45pt]stage3.north) {Supporting Activation Decisions (implicit)};
    
    \draw[|-|] (-1.5,0.5) -- (9,0.5);
    \node[] at (-1.5,0) {$t = 0$};
    \node[] at (9,0) {$t =T$};
    
    \draw[very thick] decorate [decoration={name=brace,amplitude=10pt}] {(9.25,8) -- (9.25,5.25)};
    \draw[very thick] decorate [decoration={name=brace,amplitude=10pt}] {(9.25,4.75) -- (9.25,1)};
    
    \node[anchor=west] at (10.5,5.5) {\emph{First Stage}};
    \node[anchor=west] at (10,6.625) {$z^1 \in Z^1$};
    \node[anchor=west] at (10.5,4.5) {\emph{Second Stage}};
    \node[anchor=west] at (10,2.75) {$z^2 \in \prod\limits_{E \in \mathcal{E}} Z^{2,E}(z^1)$};
    
    \end{tikzpicture}
    \caption{Schematic representation of the first- and second-stage decisions in the two-stage multicriteria formulation and the sequencing of these decisions across the planning horizon~$\mathcal{T}=\{0,1,\ldots,T\}$.}
    \label{fig:temp-sequence}
\end{figure}
In the first stage, two kinds of decisions are taken jointly.
First, \emph{preparedness decisions} strengthen the system in anticipation of adverse events. Following~\citet{faturechi_mathematical_2014}, we distinguish three complementary classes of preparedness decisions: \emph{expansive} preparedness decisions add redundancy and additional capacity that is available both under routine operations and in reaction to an adverse event (for instance, a permanent capacity expansion of a transportation route or distribution facility), \emph{reinforcing} preparedness decisions reduce the vulnerability of network components, so that a given adverse event impacts them less severely (for instance, the physical hardening of a distribution facility), and \emph{supporting} preparedness decisions establish flexible resources or capabilities that can be activated only once an adverse event has started (for instance, a contracted stand-by supplier). These decisions are subject to the \emph{preparedness condition}: their scale should enhance system resilience while remaining cost-effective.

Second, a \emph{nominal flow over time} is determined that represents the operation of the system under routine operations over a prescribed time horizon. This nominal flow may modify the system's established routine operation in order to improve resilience. Such modifications must, however, not come at the expense of routine performance. The nominal flow is therefore subject to the \emph{effectiveness condition}, which permits deviations from a prescribed \emph{reference operation} of the system only within given tolerances. Together, the preparedness decisions and the nominal flow over time form a feasible first-stage decision~$z^1 \in Z^1$.

Then, an adverse event occurs, described by time-dependent demand surges and shocks to the system's transit, supply, and storage capacities. In the second stage, upon the start of an adverse event, the system is adjusted reactively: the supporting preparedness investments procured in the first stage may be activated, and a \emph{reactive flow over time} that represents the operation of the system in reaction to the adverse event is determined for the remainder of the horizon. For each adverse event~$E$ in a given set of adverse events~$\mathcal{E}$, the activation of the supporting reserves together with the reactive flow over time forms the feasible reactive decision~$z^{2,E} \in Z^{2,E}(z^1)$ that depends on the first-stage decision~$z^1$. Collecting the reactions to all adverse events yields the second-stage decision~$z^2 = (z^{2,E})_{E \in \mathcal{E}} \in \prod_{E \in \mathcal{E}} Z^{2,E}(z^1)$.

The first- and second-stage decisions~$z = (z^1, z^2)$ are jointly evaluated through four resilience criteria. Each criterion is essentially measured through the difference in performance between the nominal flow over time and the reactive flow over time, that is, through the per-time step deviations of the operation under the adverse event, aggregated over~$\mathcal{E}$, from the operation under routine operations.
The \emph{rebound} criterion~$\Freb(z^1,z^2)$ states that the earlier the system attains and returns to routine operations, the higher its resilience.
The \emph{resistance} criterion~$\Fres(z^1,z^2)$ states that a longer phase during which a system withstands the impact of an adverse event contributes positively to resilience. 
The \emph{loss} criterion~$\Floss(z^1,z^2)$ states that a smaller performance deficit accumulated over the horizon contributes positively to resilience.
The \emph{maximum performance degradation} (MPD) criterion~$\Fmpd(z^1,z^2)$ states that the smaller the worst-case drop in performance, the higher the resilience. 
In summary, with $Z \coloneqq \Bigl\{\, z=(z^1,z^2) : z^1\in Z^1,\ z^2\in\textstyle\prod_{E\in\mathcal E}Z^{2,E}(z^1) \,\Bigr\}$ and $F(z^1,z^2) \coloneqq \bigl( \Freb(z^1,z^2),\ \Fres(z^1,z^2),\ \Floss(z^1,z^2),\ \Fmpd(z^1,z^2) \bigr) $, the \emph{two-stage multicriteria formulation} reads concisely as
\begin{align}\label{eq:two-stage-problem}
    \min_{(z^1,z^2) \in Z} \quad  & F(z^1,z^2) \tag{$P$} 
\end{align}

\subsection{Flows over Time}\label{sec:flow-over-time}
Let the system be described as a directed network $G=(V,A)$ encompassing a set of vertices $V$ and a set of arcs $A$.
In the flow over time model~\citep{Skutella2009} over a discrete time horizon $\mathcal{T} = \{0,1,\ldots, T\}$, $T \in \NN$, each arc $a \in  A$ has a nonnegative \emph{transit time} $\tau_a \in \NN$ specifying  the required amount of time for flow to travel the arc~$a$. Further, each arc $a \in  A$ has nonnegative \emph{transit capacities}~$\bar{f}_a: \mathcal{T} \rightarrow \RR_{\geq 0}$ specifying upper bounds on the amount of flow that can enter arc~$a$ per time step. Each vertex~$v \in V$ has \emph{supply capacities}~$\bar{s}_v: \mathcal{T} \rightarrow \RR_{\geq 0}$ representing the maximum amount of flow that can be generated, \emph{demands}~$\bar{d}_v: \mathcal{T} \rightarrow \RR_{\geq0}$ representing the amount of flow required, and \emph{storage capacities}~$\bar{\kappa}_v: \mathcal{T} \rightarrow \RR_{\geq 0}$ representing the maximum amount of flow that can be stored at vertex $v$ at each time step~$\theta$. Then, a \emph{flow over time}~$x = (f,s,d,u)$ is given by
\begin{itemize}
    \item functions $f_{a} : \mathcal{T} \rightarrow \RR_{\geq 0}$ adhering to the transit capacity constraint~$f_{a}(\theta) \leq \bar{f}_a(\theta)$ for all $a \in A$ and $\theta \in \mathcal{T}$, where $f_{a}(\theta)$ is understood as the rate of flow entering arc $a$ at time $\theta$,
    \item functions~$s_v : \mathcal{T} \rightarrow \RR_{\geq 0}$ satisfying the supply generation constraint~$s_v(\theta) \leq \bar{s}_v(\theta)$ for all $v \in V$ and $\theta \in \mathcal{T}$, 
    where $s_{v}(\theta)$ is understood as the rate of supply generated at vertex $v$ at time $\theta$, 
    \item functions~$d_v : \mathcal{T} \rightarrow \RR_{\geq 0}$ such that $\sum_{\xi = 0}^{\theta} d_v(\xi) \leq \sum_{\xi = 0}^{\theta} \bar{d}_v(\xi)$ for all $v \in V$ and $\theta \in \mathcal{T}$, where 
    $\sum_{\xi = 0}^{\theta} d_v(\xi)$ is understood as the cumulative amount of demand met at vertex~$v$ up to time $\theta$,
    \item (auxiliary) functions~$u_v : \mathcal{T} \rightarrow \RR$ such that $\sum_{\xi = 0}^\theta d_v(\xi) + \sum_{\xi = 0}^\theta u_v(\xi) = \sum_{\xi = 0}^\theta \bar{d}_v(\xi)$ for all $v \in V$ and  $\theta \in \mathcal{T}$, where $\sum_{\xi = 0}^{\theta} u_{v}(\xi)$ is understood as the cumulative amount of unmet demand at vertex~$v$ up to time~$\theta$.\footnote{Note that met demand is constrained cumulatively and, consequently, this formulation permits that demand arising at one time step may be served at a later time step. Accordingly, a value $u_v(\theta) < 0$ corresponds to a time step at which previously outstanding demand is served.} 
\end{itemize} 
In the following, for simplification of notation, we just write $f_{a}(\theta) = 0$ for $\theta \notin \mathcal{T}$. Further, for a vertex $v \in V$, let $\delta^+(v)$ and $\delta^-(v)$ denote the set of arcs leaving and entering vertex~$v$, respectively. 
Based on that, the excess~$\textup{ex}_v(\theta)$ of a vertex~$v \in V$ at time $\theta$ is defined as the net amount of flow that enters vertex $v$ up to time $\theta$ and thus can be understood as the amount of flow stored at vertex $v$ at time step $\theta$, i.e., 
\begin{align*}
    \textup{ex}_{v}(\theta) \coloneqq &\sum_{a \in \delta^-(v)} \left( \sum_{\xi = 0}^{\theta - \tau_a} f_{a}(\xi) \right) - \sum_{a \in \delta^+(v)} \left( \sum_{\xi = 0}^{\theta} f_{a}(\xi) \right) + \sum_{\xi = 0}^{\theta} s_v(\xi) - \sum_{\xi = 0}^{\theta} d_v(\xi).
\end{align*}
The flow $x$ is called \emph{feasible} if it satisfies the weak conservation constraints $\textup{ex}_{v}(\theta) \geq 0$ for all $\theta \in \mathcal{T}$ and meets the storage capacities $\textup{ex}_{v}(\theta) \leq \bar{\kappa}_v(\theta)$ for all $v \in V$ and $\theta \in \mathcal{T}$. The \emph{set of feasible flows} is then given by all feasible flows and denoted by~$X$. 
We consider that the system's performance under routine operations is measured by the cumulative unmet demand.
Accordingly, the \emph{performance curve in unmet demand} of $x$ over time horizon $\mathcal{T}$ is given by $P_x: \mathcal{T} \rightarrow \RR, \theta \mapsto P_x(\theta) \coloneqq \sum_{v \in V} \sum_{\xi=0}^{\theta} u_v(\xi)$.

\subsection{Adverse Events and Reactive Flows over Time}\label{sec:adverse-events-reactive-flows}

An \emph{adverse event} is described by time-dependent demand surges and by shocks to the system's transit, supply, and storage capacities.
Formally, an adverse event~$E$ is given by a start time~$\theta^E \in \mathcal{T}$ indicating the time step at which the occurrence of the adverse event becomes known, a duration~$\ell^E$ such that $\theta^E + \ell^E \leq T$, and time-dependent impacts
$
    \Delta^E(\theta) = \left( 
    \Delta^E \bar{f}_a(\theta), 
    \Delta^E \bar{s}_{v}(\theta),
    \Delta^E \bar{d}_{v}(\theta),
    \Delta^E \bar{\kappa}_{v}(\theta)
\right),
$
where
\begin{itemize}
\item $\Delta^E \bar{f}_a(\theta) \in [-\bar{f}_a(\theta), 0]$, $a \in A$, is the transit capacity reduction for arc $a$ for $\theta^E \leq \theta \leq \theta^E + \ell^E - 1$,
\item  $\Delta^E \bar{s}_{v}(\theta) \in [-\bar{s}_{v}(\theta), 0]$, $v \in V$, is the supply reduction at vertex $v$ for $\theta^E \leq \theta \leq \theta^E + \ell^E - 1$,
\item  $\Delta^E \bar{d}_v(\theta) \in [0, \infty)$, $v \in V$, is the demand surge at vertex $v$ for $\theta^E \leq \theta \leq \theta^E + \ell^E - 1$,
\item $\Delta^E \bar{\kappa}_v(\theta) \in [-\bar{\kappa}_v(\theta), 0]$, $v \in V$, is the storage capacity reduction at vertex $v$ for $\theta^E \leq \theta \leq \theta^E + \ell^E -1$.
\end{itemize}
Additionally, each adverse event~$E$ is associated with a non-negative weight~$w^E \geq 0$ which, for instance, can be interpreted as its importance. Denoting by $\mathcal{E}$ the set of considered adverse events, we normalize the weights so that $\sum_{E \in \mathcal{E}} w^E = 1$.
For simplicity, we set reductions and increments to be zero if $\theta \notin \{\theta^E,\ldots,\theta^E + \ell^E -1\}$. Then, the \emph{impacted transit capacities}, \emph{impacted supply capacities}, \emph{impacted demands} and \emph{impacted storage capacities} are given by $\bar{f}(\theta) + \Delta^E \bar{f} (\theta)$, $\bar{s}(\theta) + \Delta^E \bar{s}(\theta)$, $\bar{d}(\theta) + \Delta^E \bar{d}(\theta)$ and $\bar{\kappa}(\theta) + \Delta^E \bar{\kappa}(\theta)$, respectively.
We assume that the start time~$\theta^E$ of adverse events~$E$ cannot be anticipated. In contrast, we assume that the impact, duration, and weight of such events can often be estimated beforehand based on historical data or scenario analysis. This is feasible, for instance, in the context of earthquakes~\citep{jabbarzadeh2014dynamic} or heavy rainfall events~\citep{IPCC2021}.

\medskip

Any flow over time~$x = (f,s,d,u)$ that is initially feasible may become infeasible upon the start of an adverse event~$E$. Such infeasibility may arise, for instance, from violations of the impacted capacities or storage capacities. That is, the flow must be adjusted \emph{in reaction} to the adverse event. Given an adverse event~$E$, let $\mathcal{T}^E \coloneqq \{\theta^E, \theta^E + 1, \ldots, T\}$. Then, given a nominal flow~$x = (f,s,d,u)$, a \emph{reactive flow}~$x^E = (f^E,s^E,d^E,u^E)$ to the adverse event~$E$ with reaction horizon~$\mathcal{T}^E$ is given by
\begin{itemize}
    \item functions $f^E_a : \mathcal{T}^E \rightarrow \RR_{\geq 0}$ adhering to the impacted transit capacity constraint~$f^E_a(\theta) \leq \bar{f}_a(\theta) + \Delta^E \bar{f}_a(\theta)$ for all $\theta \in \mathcal{T}^E$, 
    \item functions~$s^E_v : \mathcal{T}^E \rightarrow \RR_{\geq 0}$ satisfying the impacted supply generation constraint~$s^E_v(\theta) \leq \bar{s}_v(\theta) + \Delta^E \bar{s}_v(\theta)$ for all $\theta \in \mathcal{T}^E$, 
    \item functions~$d^E_v : \mathcal{T}^E \rightarrow \RR_{\geq 0}$ such that $\sum_{\xi = 0}^{\theta^E - 1} d_v(\xi) + \sum_{\xi = \theta^E}^{\theta} d^E_v(\xi) \leq \sum_{\xi=0}^{\theta} \left( \bar{d}_v(\xi) + \Delta^E \bar{d}_v(\xi) \right)$ for all $\theta \in \mathcal{T}^E$, 
    \item (auxiliary) functions~$u^E_v : \mathcal{T}^E \rightarrow \RR$ such that $\sum_{\xi = 0}^{\theta^E - 1} d_v(\xi) + \sum_{\xi = \theta^E}^{\theta} d^E_v(\xi) + \sum_{\xi = 0}^{\theta^E - 1} u_v(\xi) + \sum_{\xi = \theta^E}^{\theta} u^E_v(\xi) = \sum_{\xi=0}^{\theta} \left( \bar{d}_v(\xi) + \Delta^E \bar{d}_v(\xi) \right)$ for all $\theta \in \mathcal{T}^E$,
\end{itemize}
Further, to allow flow being discarded and therefore to guarantee the existence of feasible reactive flows, we augment the vertex set $V$ with an artificial disposal vertex $v^{\text{disp}}$ with infinite storage capacities~$\bar{\kappa}_{v^{\text{disp}}}(\theta) = \infty$ 
and zero supplies and demands~$\bar{s}_{v^{\text{disp}}}(\theta) = 0$, $\bar{d}_{v^{\text{disp}}}(\theta) = 0$ for all $\theta \in \mathcal{T}$. Moreover, for every vertex~$v$ not equal to $v^{\text{disp}}$, we add an arc~$(v,v^{\text{disp}})$ to $A$ with capacity~$\bar{f}_{(v,v^{\text{disp}})}(\theta) = \infty$ for all $\theta \in \mathcal{T}$ and $\tau_{(v,v^{\text{disp}})} = 0$. Then, flow sent to $v^{\text{disp}}$ represents \emph{discarded flow} in reaction to the adverse event~$E$.
Analogously to the nominal case, the excess at vertex~$v$ at time~$\theta$ with respect to the reactive flow~$x^E$ is now given by $\textup{ex}^E_v(\theta) \coloneqq \textup{ex}_v(\theta)$ if $\theta \leq \theta^E -1$ and, otherwise
\begin{align*}
    &\textup{ex}^E_v(\theta) \coloneqq {}
    \sum_{a \in \delta^-(v)} \left(
        \sum_{\substack{\xi \leq \theta^E - 1 \\ \xi + \tau_a \leq \theta}} f_a(\xi)
        + \sum_{\substack{\xi \geq \theta^E \\ \xi + \tau_a \leq \theta}} f^E_a(\xi)
      \right) - \sum_{a \in \delta^+(v)} \left(
        \sum_{\xi = 0}^{\theta^E - 1} f_a(\xi)
        + \sum_{\xi = \theta^E}^{\theta} f^E_a(\xi)
      \right) \\
    &\qquad + \left(
        \sum_{\xi = 0}^{\theta^E - 1} s_v(\xi)
        + \sum_{\xi = \theta^E}^{\theta} s^E_v(\xi)
      \right) - \left(
        \sum_{\xi = 0}^{\theta^E - 1} d_v(\xi)
        + \sum_{\xi = \theta^E}^{\theta} d^E_v(\xi)
      \right).
\end{align*}
The reactive flow~$x^E$ is called \emph{feasible} if it satisfies the weak conservation constraints $\textup{ex}^E_v(\theta) \geq 0$ and meets the impacted storage capacities $\textup{ex}^E_v(\theta) \leq \bar{\kappa}_v(\theta) + \Delta^E \bar{\kappa}_v(\theta)$ for all $v \in V$ and all $\theta \in \mathcal{T}^E$. Note that, following the assumption that the time of occurrence of adverse events cannot be anticipated, the reactive flow coincides with the nominal flow at all time steps prior to the start time of the adverse event.
Finally, the performance curve in unmet demand of $x^E$ over time horizon~$\mathcal{T}$ is given by $P_{x^E}: \mathcal{T} \rightarrow \RR, \theta \mapsto P_{x^E}(\theta) \coloneqq \sum_{v \in V} \left( \sum_{\xi = 0}^{\theta^E - 1} u_v(\xi) + \sum_{\xi=\theta^E}^{\theta} u^E_v(\xi) \right)$.

\section{A Two-Stage Multicriteria Optimization Model to Improve Resilience}\label{sec:multiobjective}
Building on the problem setting in Section~\ref{sec:problem-formulation}, we now present the complete two-stage multicriteria optimization model. We first detail how the four resilience criteria rebound, resistance, loss, and MPD are modeled as objective functions (Section~\ref{sec:objectives}). We then specify the first-stage decision space and illustrate how the complementary conditions of preparedness and effectiveness are incorporated as constraints (Section~\ref{sec:first-stage}). Finally, we introduce the second-stage decision space of reactive flows over time, describe how the supporting preparedness decisions are activated in reaction to an adverse event, and specify how the first- and second-stage decisions are coupled (Section~\ref{sec:second-stage}). A summary of the complete model formulation can be found in~\ref{app:complete-model-formulation}.

\subsection{Objective Functions}\label{sec:objectives}
In the context of flow over time, performance under routine operations can be captured using a \emph{nominal flow over time}, while the performance under an adverse event can be captured by a \emph{reactive flow over time}. Consequently, each of the resilience criteria can be expressed in terms of the \emph{performance gap}~$g_E: \mathcal{T} \rightarrow \RR, \theta \mapsto g_E(\theta) \coloneqq P_{x^E}(\theta) - P_{x}(\theta)$, that is, the difference between the performance of a reactive flow~$x^E$ and that of the nominal flow~$x$. Since the reactive flow coincides with the nominal flow prior to the event, $g_E(\theta)$ is expressed in the model variables, for $\theta \in \mathcal{T}^E$, as
$g_E(\theta) = \sum_{v \in V} \sum_{\xi = \theta^E}^{\theta} \bigl( u^E_v(\xi) - u_v(\xi) \bigr)$,
on the basis of which we can model each resilience criteria.

The \emph{rebound} criterion states that the earlier the system attains and returns to routine operations, the higher its resilience. Correspondingly, for each adverse event~$E$, we measure the number of time steps until the nominal and reactive performance coincide again. 
To model rebound~$\Freb$ as a minimization objective, we introduce, for each $E \in \mathcal{E}$, binary variables~$r^{\mathrm{reb},E} : \mathcal{T}^E \rightarrow \{0,1\}$ together with the indicator constraints
\begin{align*}
    r^{\mathrm{reb},E}(\theta) = 1 &\Rightarrow g_E(\theta) \leq 0
\end{align*}
for all $\theta \in \mathcal{T}^E$. 
Then, the monotonicity constraints
\begin{align*}
    r^{\mathrm{reb},E}(\theta) &\geq r^{\mathrm{reb},E}(\theta - 1)
\end{align*} 
for all $\theta \in \mathcal{T}^E$ with $\theta > \theta^E$ in combination with the minimization sense of the criterion allow the rebound objective function to be represented by
\begin{align*}
    \Freb(z^1,z^2) &\coloneqq \sum_{E \in \mathcal{E}} w^E \cdot \left( \sum_{\theta \in \mathcal{T}^E} \bigl( 1 - r^{\mathrm{reb},E}(\theta) \bigr) \right).
\end{align*}

The \emph{resistance} criterion states that a longer phase during which a system withstands the impact of an adverse event contributes positively to resilience.
To model resistance~$\Fres$ as a minimization objective, we introduce, for each $E \in \mathcal{E}$, binary variables~$r^{\mathrm{res},E} : \mathcal{T}^E \rightarrow \{0,1\}$ together with the indicator constraints
\begin{align*}
    r^{\mathrm{res},E}(\theta) = 1 &\Rightarrow g_E(\theta) \leq 0
\end{align*}
for all $\theta \in \mathcal{T}^E$. Then, the monotonicity constraints
\begin{align*}
    r^{\mathrm{res},E}(\theta) &\leq r^{\mathrm{res},E}(\theta - 1)
\end{align*}
for all $\theta \in \mathcal{T}^E$ with $\theta > \theta^E$ in combination with the minimization sense of the criterion allow the resistance objective function to be represented by
\begin{align*}
    \Fres(z^1,z^2) &\coloneqq T - \sum_{E \in \mathcal{E}} w^E \cdot \left( \sum_{\theta \in \mathcal{T}^E} r^{\mathrm{res},E}(\theta) \right).
\end{align*}

The \emph{loss} criterion states that a smaller performance deficit accumulated over the horizon contributes positively to resilience. 
For each adverse event $E$, it is the area under the (non-negative part of the) performance gap. 
To model loss~$\Floss$, we introduce, for each $E \in \mathcal{E}$ and $\theta \in \mathcal{T}^E$, a non-negative variable~$r^{\mathrm{loss},E}(\theta) \geq 0$ together with the constraints
\begin{align*}
    g_E(\theta) \leq r^{\mathrm{loss},E}(\theta)
\end{align*}
for all $\theta \in \mathcal{T}^E$, which in combination with the minimization sense of the criterion allow the loss objecive function to be reprsented by
\begin{align*}
    \Floss(z^1,z^2) &\coloneqq \sum_{E \in \mathcal{E}} w^E \cdot \sum_{\theta = \theta^E}^{T}  r^{\mathrm{loss},E}(\theta).
\end{align*}

The \emph{maximum performance degradation} (MPD) criterion states that the smaller the worst-case drop in performance, the higher the resilience. For each adverse event~$E$, it is captured by the peak of the performance gap.
To model MPD~$\Fmpd$, we introduce, for each $E \in \mathcal{E}$, a non-negative variable~$r^{\mathrm{MPD},E} \geq 0$ together with the constraints
\begin{align*}
    g_E(\theta) \leq r^{\mathrm{MPD},E}
\end{align*}
for all $\theta \in \mathcal{T}^E$, which in combination with the minimization sense of the criterion allow the MPD objective function to be represented by
\begin{align*}
    \Fmpd(z^1,z^2) 
    &\coloneqq \sum_{E \in \mathcal{E}} w^E \cdot r^{\mathrm{MPD},E}.
\end{align*}

\subsection{First-Stage Decisions}\label{sec:first-stage}
Following the high-level description of Section~\ref{sec:decision-framework}, the first stage comprises the preparedness decisions and a nominal flow over time. We first formalize the three classes of preparedness decisions, expansive, reinforcing, and supporting, together with the preparedness condition that keeps their scale cost-effective. We then determine the nominal flow over time subject to the effectiveness condition that bounds its deviation from a reference operation.

\emph{Expansive preparedness decisions} refactor the network's parameters by adding redundancy and additional capacities to both nominal and reactive flows to anticipate the impact of the adverse event. To model these decisions, we introduce parameters representing the potential scale of network expansions and decision variables representing the proportion of the maximum expansion implemented. More precisely, expandable arc, supply, and storage capacities are bounded by $\bar{f}^{\mathrm{exp}}_a, \bar{s}^{\mathrm{exp}}_v, \bar{\kappa}^{\mathrm{exp}}_v \in \RR_{\geq 0}$, and decision variables $\alpha_a, \beta_v, \gamma_v \in [0,1]$ determine the implemented proportion, so that the decisions add $\bar{f}^{\mathrm{exp}}_a \alpha_a$, $\bar{s}^{\mathrm{exp}}_v \beta_v$, and $\bar{\kappa}^{\mathrm{exp}}_v \gamma_v$ to the respective capacities. Expansions to network topology, such as the addition of new arcs and new supply vertices, are, without loss of generality, included in the original network with zero initial arc and supply capacities.

\emph{Reinforcing preparedness decisions} reduce the vulnerability of components upon the start of an adverse event~\citep{Hien2020}. Variables $\zeta_a, \sigma_v, \eta_v \in [0,1]$ represent the proportion of the adverse event's impact mitigated through reinforcements into arc, supply, and storage capacity, respectively.

\emph{Supporting preparedness decisions} establish flexible resources or capabilities that become available only upon the start of an adverse event to address the immediate negative effects. In the first stage, variables $\Phi,\Pi,\Psi \in \RR_{\geq 0}$ quantify investments in additional arc, supply, and storage resources, respectively. The reserves that these investments make available are bounded by supporting capacities $\bar{f}^{\text{sup}}_a, \bar{s}^{\text{sup}}_v, \bar{\kappa}^{\text{sup}}_v\in \RR_{\geq 0}$. The activation of these reserves takes place in reaction to an adverse event and is therefore part of the second-stage decisions (Section~\ref{sec:second-stage}).

Any of these preparedness decisions is subject to costs. Let $c_{\alpha}, c_{\beta}, c_{\gamma}$ denote the cost coefficients of the expansive decisions on arc, supply, and storage capacities, $c_{\bar{f}}, c_{\bar{s}}, c_{\bar{\kappa}}$ those of the reinforcing decisions on arc, supply, and storage capacities, and $c_{\Phi}, c_{\Pi}, c_{\Psi}$ those of the supporting decisions on arc, supply, and storage resources. The total preparedness expenditure across all preparedness decisions is then given by $$R^{\mathrm{prep}}(\alpha,\beta,\gamma,\zeta,\sigma,\eta,\Phi,\Pi,\Psi) \coloneqq c_{\alpha}^{\top}\alpha + c_{\beta}^{\top}\beta + c_{\gamma}^{\top}\gamma + c_{\bar{f}}^{\top}\zeta + c_{\bar{s}}^{\top}\sigma + c_{\bar{\kappa}}^{\top}\eta + c_{\Phi}\Phi + c_{\Pi}\Pi + c_{\Psi}\Psi$$
and the \emph{preparedness condition} requires that this expenditure stay within a single joint budget $B^{\mathrm{prep}}$, 
\begin{align}\label{eq:preparedness-budget}
    R^{\mathrm{prep}}(\alpha,\beta,\gamma,\zeta,\sigma,\eta,\Phi,\Pi,\Psi) \leq B^{\mathrm{prep}}.
\end{align}
Structural restrictions may additionally couple the preparedness decisions to their activation in reaction to an adverse event, for example by enforcing integrality on certain decisions or by limiting the number of components that may receive support. As these restrictions involve the second-stage activation decisions, they are stated in Section~\ref{sec:second-stage}.

Beyond the preparedness decisions, the first stage determines a \emph{nominal flow over time}~$x = (f,s,d,u)$ that represents the operation of the system under routine operations. This nominal flow is a flow over time in the sense of Section~\ref{sec:flow-over-time}, except that it may exploit the expanded capacities provided by the expansive preparedness decisions. Accordingly, the nominal flow respects the expanded capacities, for all $a \in A$ and all $\theta \in \mathcal{T}$,
\begin{align}
    f_a(\theta) &\leq \bar{f}_a(\theta) + \bar{f}_a^{\textup{exp}} \, \alpha_a 
\end{align}
and, for all $v \in V$ and all $\theta \in \mathcal{T}$,
\begin{align}
    s_v(\theta) &\leq \bar{s}_v(\theta) + \bar{s}^{\text{exp}}_v \, \beta_v \\
    \sum_{\xi=0}^\theta d_v(\xi) &\leq \sum_{\xi=0}^\theta \bar{d}_v(\xi) \\  
    \sum_{\xi=0}^\theta d_v(\xi) + \sum_{\xi=0}^\theta u_v(\xi) &= \sum_{\xi=0}^\theta \bar{d}_v(\xi) \\
    0 \leq \textup{ex}_v(\theta) &\leq \bar{\kappa}_v(\theta) + \bar{\kappa}^{\text{exp}}_v \, \gamma_v.
\end{align}

The nominal flow is further subject to the \emph{effectiveness condition}, which permits modifications of the nominal flow~$x$ only within controlled deviations from a given \emph{reference flow over time}~$x^{\textup{ref}} = (f^{\text{ref}},s^{\text{ref}},d^{\text{ref}},u^{\text{ref}})$. Following the notion of stability~\citep{liu2022network}, let $H^{\textup{effect}}(\cdot,x^{\textup{ref}}) : X \rightarrow \RR^m_{\geq 0}$ be an affine-linear vector-valued mapping quantifying the deviation between~$x$ and~$x^{\textup{ref}}$, and let $B^{\textup{effect}} \in \RR^m_{\geq 0}$ be a vector of tolerances. Effectiveness then requires the deviation to stay within the tolerances, i.e.,
\begin{align}\label{eq:effectiveness}
    H^{\textup{effect}}(x,x^{\textup{ref}}) \leq B^{\textup{effect}}.
\end{align}
In summary, the  \emph{first-stage decision space}~$Z^1$ consists of all $z^1 = (x,\alpha,\beta,\gamma,\zeta,\sigma,\eta,\Phi,\Pi,\Psi)$ that satisfy
\eqref{eq:preparedness-budget} to \eqref{eq:effectiveness} with variable domains $f_a(\theta) \geq 0$ for all $a \in A$ and $\theta \in \mathcal{T}$, $s_v(\theta), d_v(\theta) \geq 0$ and $u_v(\theta) \in \RR$ for all $v \in V$ and $\theta \in \mathcal{T}$, $\alpha_a, \zeta_a \in [0,1]$ for all $a \in A$, $\beta_v, \gamma_v, \sigma_v, \eta_v \in [0,1]$ for all $v \in V$, and $\Phi, \Pi, \Psi \geq 0$.

\subsection{Second-Stage Decisions}\label{sec:second-stage}
Upon the start of an adverse event $E \in \mathcal{E}$ at time step~$\theta^E$, the nominal flow may become infeasible, and a reactive flow over time must be determined over the remaining horizon~$\mathcal{T}^E$. This reactive flow must respect the impacted capacities, but may benefit from first-stage preparedness decisions, as expansive decisions raise the base capacity, reinforcing decisions attenuate the impact of the event, and supporting decisions provide additional capacity that may flexible be activated.

The activation of the supporting reserves is a second-stage decision. For each adverse event~$E$, the activated amounts are given by decision variables $\phi^E_a, \pi^E_v, \psi^E_v : \mathcal{T}^E \rightarrow \RR_{\geq 0}$, which respect the supporting capacities $\bar{f}^{\text{sup}}_a, \bar{s}^{\text{sup}}_v, \bar{\kappa}^{\text{sup}}_v$ and whose total activation over the reaction horizon~$\mathcal{T}^E$ is limited by the respective invested budget, for all $a \in A$, $v \in V$, and $\theta \in \mathcal{T}^E$,
\begin{subequations}\label{eq:supporting-activation}
\begin{align}
    &&\phi^E_a(\theta) &\leq \bar{f}^{\text{sup}}_a,&&\sum_{a \in A} \sum_{\theta \in \mathcal{T}^E} \phi^E_a(\theta) \leq \Phi,&& \\
    &&\pi^E_v(\theta) &\leq \bar{s}^{\text{sup}}_v,&& \sum_{v \in V} \sum_{\theta \in \mathcal{T}^E} \pi^E_v(\theta) \leq \Pi,&& \\
    &&\psi^E_v(\theta) &\leq \bar{\kappa}^{\text{sup}}_v,&& \sum_{v \in V} \sum_{\theta \in \mathcal{T}^E} \psi^E_v(\theta) \leq \Psi.&&
\end{align}
\end{subequations}
Since the activation is summed over the reaction horizon, the supporting reserves are consumable: activating a unit of a resource at two time steps draws twice on the corresponding budget.

Given these activation decisions, the reactive flow $x^E = (f^E, s^E, d^E, u^E)$ respects the impacted capacities enhanced by preparedness. The base capacity raised by the expansive decisions, the impact attenuated by the reinforcing decisions, and the activated supporting reserves combine so that, for all $a \in A$ and all $\theta \in \mathcal{T}^E$,
\begin{align}
    &f^E_a(\theta) \leq \bigl(\bar{f}_a(\theta) + \bar{f}^{\text{exp}}_a \, \alpha_a\bigr) + (1 - \zeta_a) \, \Delta^E \bar{f}_a(\theta) + \phi^E_a(\theta)
\end{align}
and, for all $v \in V$ and all $\theta \in \mathcal{T}^E$,
\begin{align} 
    &s^E_v(\theta) \leq \bigl(\bar{s}_v(\theta) + \bar{s}^{\text{exp}}_v \, \beta_v\bigr) + (1 - \sigma_v) \, \Delta^E \bar{s}_v(\theta) + \pi^E_v(\theta)
\end{align}
\begin{align}
    0 \leq &\textup{ex}^E_v(\theta) \leq \bigl(\bar{\kappa}_v(\theta) + \bar{\kappa}^{\text{exp}}_v \, \gamma_v\bigr) + (1 - \eta_v) \, \Delta^E \bar{\kappa}_v(\theta) + \psi^E_v(\theta),
\end{align}
Moreover, the reactive met and unmet demand respect the impacted demands, for all $v \in V$ and all $\theta \in \mathcal{T}^E$,
\begin{subequations}
\begin{align}
    \sum_{\xi=0}^{\theta^E - 1} d_v(\xi) + \sum_{\xi=\theta^E}^{\theta} d^E_v(\xi) &\leq \sum_{\xi=0}^{\theta} \bigl( \bar{d}_v(\xi) + \Delta^E \bar{d}_v(\xi) \bigr)\\
    \sum_{\xi=0}^{\theta^E - 1} \bigl( d_v(\xi) + u_v(\xi) \bigr) + \sum_{\xi=\theta^E}^{\theta} \bigl( d^E_v(\xi) + u^E_v(\xi) \bigr) &= \sum_{\xi=0}^{\theta} \bigl( \bar{d}_v(\xi) + \Delta^E \bar{d}_v(\xi) \bigr).
\end{align}
\end{subequations}

Finally, additional restrictions on the preparedness decisions may couple the two stages directly, beyond the dependence of the reactive flow on the first-stage decisions. For each adverse event~$E$, these are captured by the linear constraints
\begin{align}\label{eq:coupling}
    &H^{\mathrm{prep}}\bigl(\alpha,\beta,\gamma,\zeta,\sigma,\eta,\Phi,\Pi,\Psi,\phi^E,\pi^E,\psi^E\bigr) \leq B^{\mathrm{prep,\,str}}.
\end{align}
In summary, given first-stage decisions~$z^1 \in Z^1$, the \emph{second-stage decision space} $\prod_{E \in \mathcal{E}} Z^{2,E}(z^1)$ is the product, over all adverse events, of the per-event decision spaces $Z^{2,E}(z^1)$. For a fixed adverse event~$E \in \mathcal{E}$, the set~$Z^{2,E}(z^1)$ consists of all $z^{2,E} = (x^E,\phi^E,\pi^E,\psi^E)$ that satisfy \eqref{eq:supporting-activation} to \eqref{eq:coupling} with the variable domains $f^E_a(\theta) \geq 0$ for all $a \in A$ and $\theta \in \mathcal{T}^E$, $s^E_v(\theta), d^E_v(\theta) \geq 0$ and $u^E_v(\theta) \in \RR$ for all $v \in V$ and $\theta \in \mathcal{T}^E$, $\phi^E_a(\theta) \geq 0$ for all $a \in A$ and $\theta \in \mathcal{T}^E$, and $\pi^E_v(\theta), \psi^E_v(\theta) \geq 0$ for all $v \in V$ and $\theta \in \mathcal{T}^E$.

\section{Multicriteria Optimization and Solution Algorithm}\label{sec:solution}
The two-stage problem~\eqref{eq:two-stage-problem} simultaneously minimizes four resilience criteria and is therefore a multicriteria optimization problem. In this section, we first introduce the notation and definitions used in multicriteria optimization. For a more in-depth introduction, we refer to~\citet{matthias_ehrgott_multicriteria_2005}. We then develop a decomposition of the problem into a family of bicriteria linear programs. We prove that solving all of these bicriteria problems is sufficient to solve~\eqref{eq:two-stage-problem}. Then, we establish the intractability of~\eqref{eq:two-stage-problem} by proving that, in general, superpolynomially many of these bicriteria problems must be solved and that, conversely, any exact solution method for~\eqref{eq:two-stage-problem} determines at least one distinct solution for each of them. This intractability motivates the development of a heuristic rather than an exact solution method. We therefore utilize so-called \emph{scalarizations} to obtain a heuristic approach for solving the problem. All proofs can be found in~\ref{app:proofs}.
 
\medskip
 
In the following, we call an image~$F(z)$ of a feasible solution~$z \in Z$ a \emph{feasible image}, and denote by $Y\coloneqq F(Z) = \{ F(z) : z \in Z\} \subseteq \RR^4$ the \emph{image set}. In multicriteria optimization problems, the notion of optimality is determined by an ordering relation. Let $y \leqq y'$ if and only if $y_i \le y'_i$ for all $i=1,\ldots,4$. Then, a feasible solution~$z\in Z$ \emph{dominates} another feasible solution~$z'\in Z$ if $F(z)\leqq F(z')$ and $F(z)\neq F(z')$. A solution is called \emph{efficient} if no other feasible solution dominates it, and in this case its image~$F(z)$ is called \emph{nondominated}. The \emph{efficient solution set} is denoted by $Z_E\coloneqq\{z\in Z : z \text{ is efficient}\}$ and the \emph{nondominated image set} by $Y_N\coloneqq F(Z_E)$. Since two efficient solutions may share the same image, the typical goal of multicriteria optimization is to find a set of feasible solutions $Z^*\subseteq Z$ that contains, for every nondominated image~$y\in Y_N$, at least one efficient solution~$z$ with $F(z)=y$.
 
\medskip
 
We approach this goal by exploiting the structure of the problem. To this end, we define the \emph{temporal resilience profile} $\ell=(\bm\tau^{\mathrm{res}},\bm\tau^{\mathrm{reb}})\in\prod_{E\in\mathcal E}\{0,\dots,\lvert\mathcal T^E\rvert\}^2$, and we let $\mathcal Z(\ell)$ denote the set of feasible solutions whose resistance and rebound indicators are fixed to the profile, that is, for every~$E \in \mathcal{E}$, it is $r^{\mathrm{res},E}(\theta)=1$ if and only if $\theta \in \{\theta^E,\dots,\theta^E+\tau^{\mathrm{res}}_E-1\}$ and $r^{\mathrm{reb},E}(\theta)=1$ if and only if $\theta \in \{\theta^E+\tau^{\mathrm{reb}}_E,\dots,T\}$. Then, $\mathcal Z(\ell)$ is a polyhedron and the two temporal resilience criteria are, for every $z\in\mathcal Z(\ell)$, constant with
$\Freb(\ell) = \sum_{E\in\mathcal E} w^E\,\tau^{\mathrm{reb}}_E$ and $\Fres(\ell) = T-\sum_{E\in\mathcal E} w^E\,\tau^{\mathrm{res}}_E$.
Consequently, on $\mathcal Z(\ell)$, only the two performance criteria loss and MPD vary. Since fixing a temporal resilience profile determines all integer variables in the program, the resulting problem can be defined as a linear program, which we call the \emph{residual bicriteria linear program}
\begin{align}\label{eq:residual-bi-obj}\tag{$P(\ell)$}
    \min_{z\in\mathcal Z(\ell)}\bigl(\Floss(z),\Fmpd(z)\bigr).
\end{align}
The reduction to a bicriteria linear program is attractive in two respects. First, for such a program, the nondominated set is a connected, piecewise-linear curve consisting of finitely many line segments, each spanned by two nondominated images~\citep{matthias_ehrgott_multicriteria_2005}. Hence, although in general infinite in cardinality, the nondominated set can be represented exactly by finitely many nondominated images. Second, these nondominated images can be computed by a dichotomic search that solves a sequence of linear programs, the number of which is linear in the number of these images~\citep{anejaBicriteriaTransportationProblem1979}. We refer to~\ref{app:dichotomic-search} for further details.
 
\medskip
 
Let $Z_E(\ell)\subseteq\mathcal Z(\ell)$ and $Y_N(\ell)\subseteq\RR^2$ be the residual efficient and nondominated set, respectively, and let $\mathcal L$ be the finite set of temporal resilience profiles. Associating with each $\ell\in\mathcal L$ its constant temporal image and its residual nondominated set yields the \emph{lifted image set} $\widehat{Y}(\ell)\coloneqq \bigl\{\,\bigl(\Freb(\ell),\Fres(\ell),y_{1},y_{2}\bigr): (y_{1},y_{2})\in Y_N(\ell)\,\bigr\}$. If $\mathcal Z(\ell)=\emptyset$, the residual problem is infeasible and $Y_N(\ell)=\emptyset$. Solving~\eqref{eq:residual-bi-obj} for every $\ell\in\mathcal L$ then suffices to solve~\eqref{eq:two-stage-problem}.
\begin{theorem}\label{thm:decomp}
For each $\ell\in\mathcal L$, let $Z^*(\ell)\subseteq\mathcal Z(\ell)$ contain, for each $y(\ell)\in Y_N(\ell)$, a solution $z(\ell)$ with $\bigl(\Floss(z(\ell)),\Fmpd(z(\ell))\bigr)=y(\ell)$. Then $\bigcup_{\ell\in\mathcal L}Z^*(\ell)$ contains, for each $y\in Y_N$, an efficient solution $z$ with $F(z)=y$.
\end{theorem}
 
Following \cite{eichfelderTestInstanceGenerator2024}, the relevant characteristic to measure the \emph{tractability} of a mixed-integer problem such as~\eqref{eq:two-stage-problem} is the number of temporal resilience profiles that yield distinct nondominated images. In the following, we call a profile $\ell\in\mathcal L$ \emph{efficient} if its lifted image set contains a nondominated image, that is, $\widehat Y(\ell)\cap Y_N\neq\emptyset$. The next theorem shows that the number of efficient temporal resilience profiles can grow superpolynomially in the encoding length.
\begin{theorem}\label{thm:intractable}
	For every even $K \in \NN$ there is an instance of~\eqref{eq:two-stage-problem} with $\lvert\mathcal E\rvert = K$ adverse events, $K+1$ vertices and $K$ arcs that admits $\binom{K}{K/2} \geq 2^{K}/(K+1)$ efficient temporal resilience profiles whose lifted image sets each contain a distinct nondominated image of~\eqref{eq:two-stage-problem}. Since the instance is encoded in $O(K^2)$ bits, the number of efficient temporal resilience profiles is not bounded by any polynomial in the encoding length, and the problem is intractable.
\end{theorem}
As a consequence, any exact method that determines the nondominated set must, on such instances, enumerate at least $\binom{K}{K/2}$ temporal resilience profiles.
This motivates the design of a heuristic that computes a subset of the efficient set together with an enclosure of the nondominated set.

The first phase samples $H$ well-spaced weight vectors $\lambda\in\RR^4$, $\lambda_i>0$, on the unit simplex using the Das--Dennis method~\citep{dasNormalBoundaryIntersectionNew1998} and solves the weighted sum problem $$\min_{z\in Z}\sum_{i=1}^{4}\lambda_i F_i(z)$$ for each. It is known that every optimal solution of a weighted sum problem is efficient. Weighted sums alone, however, systematically miss a portion of~$Y_N$~\citep{matthias_ehrgott_multicriteria_2005}.

The second phase refines the set using enclosures. Let $\RR^4_{\geqq}$ denote the nonnegative orthant. An \emph{enclosure} is a pair of a lower bound set~$L\subseteq\RR^4$ and an upper bound set~$U\subseteq Y$ with $$Y_N \subseteq (L+\RR^4_{\geqq})\cap(U-\RR^4_{\geqq}) = \bigcup_{l\in L,\,u\in U}[l,u],$$ whose nonempty boxes~$[l,u]$ are the \emph{search zones}~\citep{eichfelder_hybrid_2024,dachert_efficient_2017}. Repeatedly, a search zone of maximal width~$\max_{i}(u_i-l_i)$ is selected and the \emph{augmented weighted Tchebycheff problem} $$\min_{z\in Z}\max_{i}\lambda_i\lvert F_i(z)-r_i\rvert + \mu\sum_{i}\lambda_i F_i(z)$$ is solved with~$r=l$ and $\lambda_i=1/(u_i-l_i)$ if $u_i>l_i$ and $\lambda_i=1/\min_{j:\,u_j>l_j}(u_j-l_j)$ otherwise. Every optimal solution is again efficient, and it is known that any efficient solution is optimal for an augmented weighted Tchebycheff problem with suitably chosen $\lambda,\mu$, and $r$. If $F(z)\in[l,u]$, both bound sets are refined; otherwise the zone contains no nondominated image and is discarded, updating~$L$. The phase terminates once every search zone falls below a prescribed width or a time limit is reached. The procedure is summarized in Algorithm~\ref{alg:enclosure}.

\begin{algorithm}[t]
\caption{Scalarization-based enclosure heuristic for the nondominated set~$Y_N$}
\label{alg:enclosure}
\KwIn{two-stage problem~\eqref{eq:two-stage-problem}; number of weight vectors~$H$; width tolerance~$\overline{w}\in\RR^4_{>0}$; time limit~$\Theta$; dichotomic routine~$\textsc{Dicho}$}
\KwOut{an archive~$Z^*\subseteq Z_E$ with images~$F(Z^*)$ and an enclosure~$(L,U)$}
$Z^*\gets\emptyset$; initialize the enclosure~$(L,U)$\;
 
\ForEach{Das--Dennis weight vector~$\lambda$ with~$H$ weight vectors}{
    solve the weighted sum problem for~$\lambda$; let~$z$ be an optimal solution\;
    update~$Z^*$ by~$z$ and refine~$(L,U)$\;
}
\While{some search zone~$[l,u]$ satisfies~$u_i-l_i>\overline{w}_i$ for some $i\in\{1,\ldots,4\}$ \textnormal{\textbf{and}}~$\Theta$ is not reached}{
    select a search zone~$[l,u]$ of maximal width\;
    solve the augmented weighted Tchebycheff problem with~$r=l$ and~$\lambda_i=1/(u_i-l_i)$; let~$z$ be an optimal solution\;
    \eIf{$F(z)\in[l,u]$}{
        update~$Z^*$ by~$z$ and refine~$(L,U)$\;
    }{
        discard~$[l,u]$ and update the lower bound set~$L$\;
    }
}
$Z^*\gets(Z^*)_E$\;
\Return{$Z^*$, $F(Z^*)$, and~$(L,U)$}\;
\end{algorithm}

\section{Computational Experiments}\label{sec:experiments}


\subsection{Instances and Computational Setup}
To analyze the structural properties of the set of efficient solutions and their nondominated images with respect to the existence of conflicts, their local marginal trade-off rates and 
the relation between each resilience criterion and preparedness decisions, we conduct numerical studies on a set of test instances.
The combination of flows over time, adverse events, and preparedness and effectiveness conditions in a multicriteria setting has not yet been studied as explicitly outlined in~\cite{sharkey_search_2020}.
Hence, we generate static minimum cost flow instances using the NETGEN grid algorithm of~\cite{klingman_netgen_1974} which are expanded over a horizon $T\in\{100,150,200,250,300\}$ and tightened so as to yield highly utilized networks. For each instance, $\lvert\mathcal E\rvert=10$ adverse events are sampled following~\cite{Eshghali2023}. The full generation procedure and a summary of the characteristics of the resulting set of instances~$\mathcal{I}$ can be found in~\ref{app:instances}. 
The heuristic has been implemented using a Python 3.12.8 environment and using Gurobi~12.0.3 as the underlying solver for the scalarized two-stage mixed-integer programs. Computations are performed on a machine with 16-core CPU and 128 GB of RAM in a Red Hat Enterprise Linux environment. The weighted-sum phase generates $H = 56$ well-spaced weight vectors using the method from~\cite{dasNormalBoundaryIntersectionNew1998}, and the enclosure phase terminates once a time limit of 24h is reached or every search zone is narrower than $\overline{w} = (1,1,5,2)$, a width tolerance chosen to reflect the differing scales of rebound, resistance, loss, and MPD, respectively. In the following, for each instance $I \in \mathcal{I}$, we denote by $Z^*_I$ the returned set of solutions.

\subsection{Results - Resilience Criteria}\label{sec:results-multiple-instances}
To explore which criteria vary across the efficient set and are thus in conflict, we consider, for each instance~$I$, the distribution of each criterion's~$i \in \{\mathrm{res}, \mathrm{reb}, \mathrm{loss}, \mathrm{MPD}\}$ normalized range over the returned sets of solutions~$Z^*_I$:
$$
\rhorange_i (I) \coloneqq \dfrac{\max_{z^*}F^i(z^*) -\min_{z^*}F^i(z^*)}{\max_{z^*}F^i (z^*)},
$$
if $\max_{z^*}F^i(z^*)>0$ and $\rhorange_i (I) \coloneqq 0$ otherwise.
Figure~\ref{fig:instances-type-ranges} (Left) reports the normalized range $\rho_i^{\mathrm{range}}(I)$ of each criterion~$i$ across the instances~$\mathcal{I}$. The rebound duration attains by far the widest range (median $0.21$, with maximum normalized range of $0.72$), whereas the remaining three criteria vary comparatively little: loss and MPD attain medians of approximately $0.09$ each (maximum normalized range of $0.31$ and $0.32$), and the resistance duration attains the narrowest range of all (median $0.03$, maximum normalized range of $0.32$). Across the computed set of solutions, rebound thus spans a substantially larger portion of its attainable range than the other criteria, so that solutions that are near-indistinguishable in loss, MPD, and resistance may still differ substantially in rebound.

\emph{On the instances studied, a target rebound value is therefore the most consequential input for the choice of a solution. Rebound is the only criterion whose value varies substantially across the efficient set. Therefore, formulating a rebound target value resolves the criterion along which candidate solutions genuinely differ, whereas targets on loss, MPD, and resistance have almost no discriminatory power among candidate solutions.}

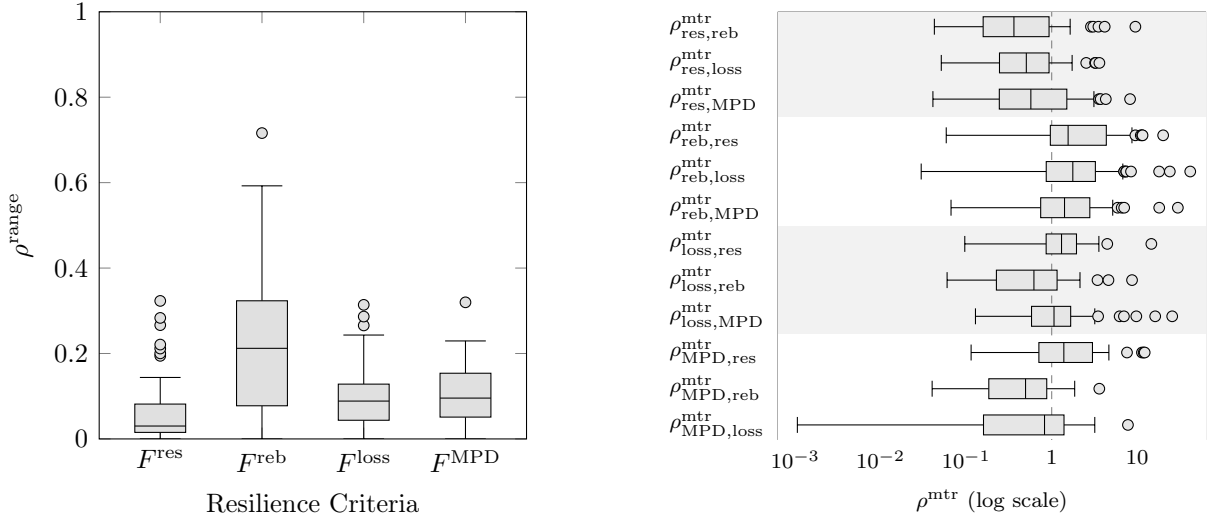
\begin{figure}[tb]
    \centering
    \begin{subfigure}[t]{0.49\textwidth}
        \centering
\begin{tikzpicture}
\begin{axis}[
    boxplot/draw direction=y,
    ylabel={$\rho^{\mathrm{range}}$},
    ymin=0, ymax=1,
    ytick={0,0.2,0.4,0.6,0.8,1},
    xtick={1,2,3,4},
    xlabel={Resilience Criteria},
    xticklabels={$F^{\mathrm{res}}$,$F^{\mathrm{reb}}$,$F^{\mathrm{loss}}$,$F^{\mathrm{MPD}}$},
    height=11cm, width=11cm,
    boxplot={box extend=0.5},
    font=\footnotesize, 
    every axis plot/.append style={black, fill=black!12},
    cycle list={{}},
    scale = 0.6, 
    ]
        \addplot+[boxplot prepared={
            lower whisker=0.0000, lower quartile=0.0153,
            median=0.0302,
            upper quartile=0.0816, upper whisker=0.1439,
        }] table[row sep=\\, y index=0] {%
        0.1945\\ 0.2003\\ 0.2116\\ 0.2209\\ 0.2664\\ 0.2833\\ 0.3231\\};
        \addplot+[boxplot prepared={
            lower whisker=0.0000, lower quartile=0.0776,
            median=0.2121,
            upper quartile=0.3234, upper whisker=0.5925,
        }] table[row sep=\\, y index=0] {%
        0.7160\\};
        \addplot+[boxplot prepared={
            lower whisker=0.0000, lower quartile=0.0437,
            median=0.0886,
            upper quartile=0.1283, upper whisker=0.2432,
        }] table[row sep=\\, y index=0] {%
        0.2657\\ 0.2865\\ 0.3139\\};
        \addplot+[boxplot prepared={
            lower whisker=0.0000, lower quartile=0.0511,
            median=0.0956,
            upper quartile=0.1538, upper whisker=0.2294,
        }] table[row sep=\\, y index=0] {%
        0.3198\\};
\end{axis}
\end{tikzpicture}
    \end{subfigure}
    \begin{subfigure}[t]{0.49\textwidth}
        \centering
    \begin{tikzpicture}
\begin{axis}[
    boxplot/draw direction=x, y dir=reverse,
    xlabel={$\rho^{\mathrm{mtr}}$ (log scale)},
    xmin=-3.2, xmax=1.8,
    xtick={-3,-2,-1,0,1},
    xticklabels={$10^{-3}$,$10^{-2}$,$10^{-1}$,$1$,$10$},
    ymin=0.6, ymax=12.4,
    ytick={1,2,3,4,5,6,7,8,9,10,11,12},
    yticklabels={%
      $\rho^{\mathrm{mtr}}_{\mathrm{res},\mathrm{reb}}$,%
      $\rho^{\mathrm{mtr}}_{\mathrm{res},\mathrm{loss}}$,%
      $\rho^{\mathrm{mtr}}_{\mathrm{res},\mathrm{MPD}}$,%
      $\rho^{\mathrm{mtr}}_{\mathrm{reb},\mathrm{res}}$,%
      $\rho^{\mathrm{mtr}}_{\mathrm{reb},\mathrm{loss}}$,%
      $\rho^{\mathrm{mtr}}_{\mathrm{reb},\mathrm{MPD}}$,%
      $\rho^{\mathrm{mtr}}_{\mathrm{loss},\mathrm{res}}$,%
      $\rho^{\mathrm{mtr}}_{\mathrm{loss},\mathrm{reb}}$,%
      $\rho^{\mathrm{mtr}}_{\mathrm{loss},\mathrm{MPD}}$,%
      $\rho^{\mathrm{mtr}}_{\mathrm{MPD},\mathrm{res}}$,%
      $\rho^{\mathrm{mtr}}_{\mathrm{MPD},\mathrm{reb}}$,%
      $\rho^{\mathrm{mtr}}_{\mathrm{MPD},\mathrm{loss}}$},
    yticklabel style={anchor=west, xshift=-45},
    font=\scriptsize, 
    height=11cm, width=11cm,
    every axis plot/.append style={black, fill=black!10, boxplot/box extend=0.55},
    cycle list={{}},
    scale = 0.6, transform shape
    ]

    \fill[black!5]  (axis cs:-3.2,0.5)  rectangle (axis cs:1.8,3.5);   
    \fill[black!0]  (axis cs:-3.2,3.5)  rectangle (axis cs:1.8,6.5);   
    \fill[black!5]  (axis cs:-3.2,6.5)  rectangle (axis cs:1.8,9.5);   
    \fill[black!0]  (axis cs:-3.2,9.5)  rectangle (axis cs:1.8,12.5);  

    \draw[dashed,gray] (axis cs:0,0.5) -- (axis cs:0,12.5);
    \addplot+[boxplot prepared={draw position=1, lower whisker=-1.373, lower quartile=-0.802, median=-0.442, upper quartile=-0.031, upper whisker=0.216}] table[row sep=\\,y index=0]{0.459\\ 0.486\\ 0.547\\ 0.620\\ 0.979\\};
    \addplot+[boxplot prepared={draw position=2, lower whisker=-1.292, lower quartile=-0.611, median=-0.297, upper quartile=-0.031, upper whisker=0.239}] table[row sep=\\,y index=0]{0.404\\ 0.507\\ 0.515\\ 0.558\\};
    \addplot+[boxplot prepared={draw position=3, lower whisker=-1.390, lower quartile=-0.613, median=-0.244, upper quartile=0.177, upper whisker=0.494}] table[row sep=\\,y index=0]{0.554\\ 0.573\\ 0.633\\ 0.917\\};
    \addplot+[boxplot prepared={draw position=4, lower whisker=-1.234, lower quartile=-0.016, median=0.192, upper quartile=0.639, upper whisker=0.939}] table[row sep=\\,y index=0]{0.979\\ 1.042\\ 1.058\\ 1.063\\ 1.304\\};
    \addplot+[boxplot prepared={draw position=5, lower whisker=-1.527, lower quartile=-0.064, median=0.247, upper quartile=0.511, upper whisker=0.832}] table[row sep=\\,y index=0]{0.847\\ 0.848\\ 0.870\\ 0.874\\ 0.928\\ 1.257\\ 1.382\\ 1.621\\};
    \addplot+[boxplot prepared={draw position=6, lower whisker=-1.178, lower quartile=-0.129, median=0.150, upper quartile=0.446, upper whisker=0.714}] table[row sep=\\,y index=0]{0.772\\ 0.819\\ 0.848\\ 1.258\\ 1.476\\};
    \addplot+[boxplot prepared={draw position=7, lower whisker=-1.017, lower quartile=-0.065, median=0.115, upper quartile=0.290, upper whisker=0.552}] table[row sep=\\,y index=0]{0.650\\ 1.166\\};
    \addplot+[boxplot prepared={draw position=8, lower whisker=-1.224, lower quartile=-0.647, median=-0.208, upper quartile=0.063, upper whisker=0.331}] table[row sep=\\,y index=0]{0.536\\ 0.664\\ 0.940\\};
    \addplot+[boxplot prepared={draw position=9, lower whisker=-0.892, lower quartile=-0.236, median=0.027, upper quartile=0.222, upper whisker=0.504}] table[row sep=\\,y index=0]{0.543\\ 0.795\\ 0.845\\ 0.990\\ 1.211\\ 1.410\\};
    \addplot+[boxplot prepared={draw position=10, lower whisker=-0.943, lower quartile=-0.150, median=0.142, upper quartile=0.477, upper whisker=0.669}] table[row sep=\\,y index=0]{0.881\\ 1.056\\ 1.076\\ 1.088\\};
    \addplot+[boxplot prepared={draw position=11, lower whisker=-1.399, lower quartile=-0.736, median=-0.306, upper quartile=-0.058, upper whisker=0.269}] table[row sep=\\,y index=0]{0.557\\};
    \addplot+[boxplot prepared={draw position=12, lower whisker=-2.976, lower quartile=-0.798, median=-0.084, upper quartile=0.144, upper whisker=0.504}] table[row sep=\\,y index=0]{0.892\\};
\end{axis}
\end{tikzpicture}
    \end{subfigure}
    \caption{(Left) Normalized range of each resilience criterion across instances. (Right) Distribution across instances of the per-instance median local marginal trade-off rate for each ordered pair of resilience criteria (log scale).}
    \label{fig:instances-type-tradeoff}
    \label{fig:instances-type-ranges}
\end{figure}

\medskip

Next, we explore the local marginal trade-off rates between pairs of criteria. Intuitively, the marginal trade-off rate of criterion~$j$ with respect to criterion~$i$ measures how much~$j$ must be degraded to improve~$i$ by a small amount. To quantify it, we first render the criteria comparable via a min--max normalization within each instance~$I$. We then estimate, for each solution~$z^* \in Z^*_I$, its local marginal trade-off rates within a neighborhood, and summarize these over the computed set of solutions~$Z^*_I$ by their median.
More precisely, we set for $z^* \in Z^*_I$
\begin{align*}
   \widetilde F^{i}(z^*) \;=\;
   \frac{F^{i}(z^*)-\min_{\tilde z^*\in Z^*}F^{i}(\tilde z^*)}
        {\max_{\tilde z^* \in Z^*}F^{i}(\tilde z^*)-\min_{\tilde z^*\in Z^*}F^{i}(\tilde z^*)},
\end{align*}
with the convention that a criterion that is constant on~$Z^*_I$ is set to zero and contributes no trade-off.  Then, for each solution~$z^*$, we consider its $k=\lceil \sqrt{\lvert Z^*_I\rvert} \rceil$ nearest neighbors with respect to the Euclidean distance in the normalized criterion space.
For an ordered pair of criteria~$(i,j)$, a solution $z^*\in Z^*_I$  and one if its neighbors~$\tilde z^*$ for which $\widetilde F^{i}(\tilde z^*)<\widetilde F^{i}(z^*)$ holds true, the \emph{local marginal trade-off rate} of~$j$ with respect to~$i$ is then defined by
\begin{equation}\label{eq:mtr}
   \rho^{\mathrm{mtr}}_{ij}(z^*,\tilde z^*)
   \;=\;
   \frac{\widetilde F^{j}(\tilde z^*)-\widetilde F^{j}(z^*)}
        {\widetilde F^{i}(z^*)-\widetilde F^{i}(\tilde z^*)} ,
\end{equation}
For each pair of criteria $(i,j)$, the \emph{marginal trade-off rate}~$\rho^{\mathrm{mtr}}_{ij}(I)$ of an instance~$I$ is then defined as the median among all local marginal trade-off rates of~$j$ with respect to~$i$. 

Figure~\ref{fig:instances-type-tradeoff} (Right) reports, for each ordered pair of criteria $(i,j)$, the distribution of the per-instance marginal trade-off rate $\rho^{\mathrm{mtr}}_{ij}(I)$, $I \in \mathcal{I}$.
At the median~$\bar\rho^{\mathrm{mtr}}_{ij}$ of these per-instance values, resistance exchanges most favorably against every other criterion ($\bar\rho^{\mathrm{mtr}}_{\mathrm{res},\mathrm{reb}} = 0.36$, $\bar\rho^{\mathrm{mtr}}_{\mathrm{res},\mathrm{loss}} = 0.50$, $\bar\rho^{\mathrm{mtr}}_{\mathrm{res},\mathrm{MPD}} = 0.57$).
In contrast, rebound exchanges least favorably: the three largest median rates all concern rebound and each exceeds one ($\bar\rho^{\mathrm{mtr}}_{\mathrm{reb},\mathrm{loss}} = 1.77$, $\bar\rho^{\mathrm{mtr}}_{\mathrm{reb},\mathrm{res}} = 1.56$, $\bar\rho^{\mathrm{mtr}}_{\mathrm{reb},\mathrm{MPD}} = 1.41$).
Loss exchanges unfavorably against resistance ($\bar\rho^{\mathrm{mtr}}_{\mathrm{loss},\mathrm{res}} = 1.30$) and MPD ($\bar\rho^{\mathrm{mtr}}_{\mathrm{loss},\mathrm{MPD}} = 1.06$), but favorably against rebound ($\bar\rho^{\mathrm{mtr}}_{\mathrm{loss},\mathrm{reb}} = 0.62$).
MPD exchanges unfavorably against resistance ($\bar\rho^{\mathrm{mtr}}_{\mathrm{MPD},\mathrm{res}} = 1.39$) and favorably against both loss and rebound ($\bar\rho^{\mathrm{mtr}}_{\mathrm{MPD},\mathrm{loss}} = 0.82$, $\bar\rho^{\mathrm{mtr}}_{\mathrm{MPD},\mathrm{reb}} = 0.49$). 
These median tendencies notwithstanding, for almost every ordered pair of criteria the interquartile range contains one, and for the remaining pairs the whiskers reach past it. No pair of criteria therefore exchanges uniformly across instances: for each pair, there are instances where the exchange is cheaper than proportional and instances where it is costlier. Whether a given trade-off is favorable or costly is thus not a property of the criterion pair but of the individual instance.
Moreover, the whiskers and outliers of a given pair span up to three orders of magnitude. Hence, not only the direction of a trade-off is instance-dependent but also its magnitude, i.e., even where the sign of the exchange is known, the rate itself varies across instances by orders of magnitude. 

\emph{Taken together, the ordering of the median trade-off rates indicates where improvements are cheapest to obtain, namely by improving resistance compared to other criteria, and where they are most expensive, namely by improving rebound compared to other criteria. This ordering, however, holds only at the median. Therefore, across the instances studied, the trade-off rate between any pair of criteria generally cannot be represented by a single scalar. Thus, from a decision-maker's perspective, an acceptable trade-off rate cannot be committed to in advance through a scalar index, but only in association with an instance-specific computed set of nondominated images. Prescribing a solution thus requires explicit multicriteria analyses, as only this makes available the instance-specific trade-off rates against which any acceptable rate must be weighed.}

\subsection{Results - Decisions and Resilience Criteria}
Next, we explore how the composition of the preparedness budget relates to the resilience criteria. 
To this end, we analyze whether committing a larger share of the budget to a class of preparedness decisions can be associated with better or worse values of a resilience criterion across the computed set of solutions. To quantify this, we first express each solution's budget as the shares spent on the
three classes, then measure, within each instance, the monotone association between each share and each criterion by a Spearman rank correlation, and summarize these across instances.

More precisely, for a solution~$z^*$, let $R^{\mathrm{exp}} \coloneqq c_{\alpha}^{\top}\alpha + c_{\beta}^{\top}\beta + c_{\gamma}^{\top}\gamma$, $R^{\mathrm{rei}} \coloneqq c_{\bar{f}}^{\top}\zeta + c_{\bar{s}}^{\top}\sigma + c_{\bar{\kappa}}^{\top}\eta$, and $R^{\mathrm{sup}} \coloneqq c_{\Phi}\Phi + c_{\Pi}\Pi + c_{\Psi}\Psi$ denote the budget spent on expansive, reinforcing, and supporting preparedness decisions, respectively. 
Further, let~$b_{k}(z^*) \;\coloneqq\; \frac{R^{k}(z^*)}{R^{\mathrm{prep}}(z^*)}$, $k \in \{\mathrm{exp}, \mathrm{rei}, \mathrm{sup}\}$ be the corresponding budget shares.
For each instance~$I \in \mathcal{I}$, each class~$k \in \{\mathrm{exp}, \mathrm{rei}, \mathrm{sup}\}$, and each criterion~$i \in \{\mathrm{res}, \mathrm{reb}, \mathrm{loss}, \mathrm{MPD}\}$, we then compute the Spearman rank correlation $\rhocorr_{k,i}(I)$ between the shares $\bigl(b_{k}(z^*)\bigr)_{z^* \in Z^*_I}$ and the criterion values $\bigl(F^{i}(z^*)\bigr)_{z^* \in Z^*_I}$.
Note that, since all four criteria are stated as minimization criteria, a negative correlation indicates that a larger budget share improves the criterion and a positive correlation that it degrades it. Since the three budget shares are compositional and sum to one, these associations describe the relative emphasis of the investment rather than three independent adjustable quantities.

Figure~\ref{fig:decision-spearman} reports the distribution of $\rhocorr_{k,i}(I)$, $I \in \mathcal{I}$.
\begin{figure}[tb]
    \begin{subfigure}[t]{0.49\textwidth}
    \vspace{0pt}
    \centering
    \begin{tikzpicture}
\begin{axis}[
    boxplot/draw direction=x, y dir=reverse,
    xlabel={$\rho^{\mathrm{corr}}$},
    xmin=-1.18, xmax=1.18,
    xtick={-1,-0.5,0,0.5,1},
    xticklabels={$-1$,$-0.5$,$0$,$0.5$,$1$},
    ymin=0.6, ymax=12.4,
    ytick={1,2,3,4,5,6,7,8,9,10,11,12},
    yticklabels={%
      $\rho^{\mathrm{corr}}_{\mathrm{exp},\mathrm{reb}}$,%
      $\rho^{\mathrm{corr}}_{\mathrm{exp},\mathrm{res}}$,%
      $\rho^{\mathrm{corr}}_{\mathrm{exp},\mathrm{loss}}$,%
      $\rho^{\mathrm{corr}}_{\mathrm{exp},\mathrm{MPD}}$,%
      $\rho^{\mathrm{corr}}_{\mathrm{rei},\mathrm{reb}}$,%
      $\rho^{\mathrm{corr}}_{\mathrm{rei},\mathrm{res}}$,%
      $\rho^{\mathrm{corr}}_{\mathrm{rei},\mathrm{loss}}$,%
      $\rho^{\mathrm{corr}}_{\mathrm{rei},\mathrm{MPD}}$,%
      $\rho^{\mathrm{corr}}_{\mathrm{sup},\mathrm{reb}}$,%
      $\rho^{\mathrm{corr}}_{\mathrm{sup},\mathrm{res}}$,%
      $\rho^{\mathrm{corr}}_{\mathrm{sup},\mathrm{loss}}$,%
      $\rho^{\mathrm{corr}}_{\mathrm{sup},\mathrm{MPD}}$},
    yticklabel style={anchor=west, xshift=-42},
    font=\scriptsize,
    height=11cm, width=11cm,
    every axis plot/.append style={black, fill=black!10, boxplot/box extend=0.55},
    cycle list={{}},
    scale=0.6, transform shape
    ]

    \fill[black!5]  (axis cs:-1.18,0.5)  rectangle (axis cs:1.18,4.5);   
    \fill[black!0]  (axis cs:-1.18,4.5)  rectangle (axis cs:1.18,8.5);   
    \fill[black!5]  (axis cs:-1.18,8.5)  rectangle (axis cs:1.18,12.5);  

    \draw[dashed,gray] (axis cs:0,0.5) -- (axis cs:0,12.5);

    \addplot+[boxplot prepared={draw position=1, lower whisker=-1.000, lower quartile=-0.976, median=-0.897, upper quartile=-0.761, upper whisker=-0.561}] table[row sep=\\,y index=0]{-0.434\\ -0.416\\ -0.212\\ -0.150\\ -0.130\\ 0.200\\ 0.373\\ 0.485\\ 1.000\\};
    \addplot+[boxplot prepared={draw position=2, lower whisker=-1.000, lower quartile=-0.118, median=0.169, upper quartile=0.480, upper whisker=1.000}] coordinates {};
    \addplot+[boxplot prepared={draw position=3, lower whisker=-1.000, lower quartile=-0.796, median=-0.333, upper quartile=0.397, upper whisker=1.000}] coordinates {};
    \addplot+[boxplot prepared={draw position=4, lower whisker=0.045, lower quartile=0.485, median=0.812, upper quartile=0.961, upper whisker=1.000}] table[row sep=\\,y index=0]{-1.000\\ -0.772\\ -0.690\\ -0.313\\};
    \addplot+[boxplot prepared={draw position=5, lower whisker=-0.024, lower quartile=0.463, median=0.761, upper quartile=0.871, upper whisker=1.000}] table[row sep=\\,y index=0]{-1.000\\ -0.500\\ -0.484\\ -0.311\\};
    \addplot+[boxplot prepared={draw position=6, lower whisker=-0.746, lower quartile=-0.070, median=0.284, upper quartile=0.498, upper whisker=1.000}] coordinates {};
    \addplot+[boxplot prepared={draw position=7, lower whisker=-1.000, lower quartile=-0.738, median=0.182, upper quartile=0.596, upper whisker=0.974}] coordinates {};
    \addplot+[boxplot prepared={draw position=8, lower whisker=-1.000, lower quartile=-0.941, median=-0.810, upper quartile=-0.596, upper whisker=-0.254}] table[row sep=\\,y index=0]{0.111\\ 0.448\\ 0.688\\};
    \addplot+[boxplot prepared={draw position=9, lower whisker=-1.000, lower quartile=-0.296, median=0.373, upper quartile=0.862, upper whisker=1.000}] coordinates {};
    \addplot+[boxplot prepared={draw position=10, lower whisker=-1.000, lower quartile=-0.801, median=-0.533, upper quartile=-0.296, upper whisker=0.336}] table[row sep=\\,y index=0]{0.628\\ 0.638\\ 0.811\\ 1.000\\ 1.000\\};
    \addplot+[boxplot prepared={draw position=11, lower whisker=-0.731, lower quartile=0.139, median=0.600, upper quartile=0.817, upper whisker=1.000}] table[row sep=\\,y index=0]{-1.000\\ -1.000\\ -0.947\\ -0.941\\};
    \addplot+[boxplot prepared={draw position=12, lower whisker=-1.000, lower quartile=-0.744, median=-0.062, upper quartile=0.596, upper whisker=1.000}] coordinates {};
\end{axis}
\end{tikzpicture}
    \end{subfigure}
    \begin{subfigure}[t]{0.49\textwidth}
    \vspace{0pt}
    \centering
    \resizebox{\textwidth}{!}{
    \begin{tabular}{llrrr}
    \toprule
    Class & Criterion & Median $\rho$ & Improve (\%) & Degrade (\%) \\
    \midrule 
    Expansive & Rebound      & $-0.90$ & 92 &  8 \\
              & Resistance   & $+0.17$ & 38 & 62 \\
              & Loss         & $-0.33$ & 65 & 33 \\
              & MPD          & $+0.81$ &  8 & 92 \\
    \midrule
    Reinforcements & Rebound      & $+0.76$ &  9 & 89 \\
                   & Resistance   & $+0.28$ & 26 & 74 \\
                   & Loss         & $+0.18$ & 43 & 55 \\
                   & MPD          & $-0.81$ & 94 &  6 \\
    \midrule
    Supporting & Rebound      & $+0.37$ & 32 & 68 \\
               & Resistance   & $-0.53$ & 84 & 16 \\
               & Loss         & $+0.60$ & 23 & 75 \\
               & MPD          & $-0.06$ & 53 & 47 \\
    \bottomrule
    \end{tabular}%
    
    }\vspace{1.1cm}
    \end{subfigure}
    \caption{(Left) Distribution across instances of the per-instance Spearman rank correlation between the budget share of each class of preparedness decisions and each resilience criterion (all criteria minimized; negative values indicate that a larger share improves the criterion).
    (Right) Per-instance Spearman rank correlation between the budget share of each class of preparedness decisions (expansive, reinforcements, supporting) and each resilience criterion, aggregated across all instances. All criteria are formulated as minimization objectives, so a negative $\rho$ indicates that a larger budget share improves the criterion. ``Improve'' and ``Degrade'' report the share of instances with $\rho<0$ and $\rho>0$, respectively; the two shares need not sum to $100\,\%$, as instances with a vanishing coefficient are counted in neither. }
    \label{fig:decision-spearman}
    \label{tab:decision-spearman}
\end{figure}
Opposing roles of expansive and reinforcing preparedness decisions can be identified. A larger share invested in expansive preparedness decisions is associated with a better rebound in $92\,\%$ of instances (median $\bar\rho^{\mathrm{corr}}_{\mathrm{exp},\mathrm{reb}} = -0.90$) but with a worse MPD in $92\,\%$ ($\bar\rho^{\mathrm{corr}}_{\mathrm{exp},\mathrm{MPD}} = 0.81$). In contrast, a larger share invested in reinforcing preparedness decisions is associated with a better MPD in $94\,\%$ of instances ($\bar\rho^{\mathrm{corr}}_{\mathrm{rei},\mathrm{MPD}} = -0.81$) but with a worse rebound in $89\,\%$ ($\bar\rho^{\mathrm{corr}}_{\mathrm{rei},\mathrm{reb}} = 0.76$). Hence, across the efficient sets computed for the instances, expansive preparedness decisions tend to improve rebound at the cost of MPD, whereas reinforcing preparedness decisions tend to improve MPD at the cost of rebound.
Further, a larger share invested in supporting preparedness decisions is associated with better resistance in $84\,\%$ of instances ($\bar\rho^{\mathrm{corr}}_{\mathrm{sup},\mathrm{res}} = -0.53$), whereas a larger share invested in reinforcing preparedness decisions is associated with worse one in $74\,\%$ ($\bar\rho^{\mathrm{corr}}_{\mathrm{rei},\mathrm{res}} = 0.28$). Expansive preparedness decisions show no consistent relation to resistance ($\bar\rho^{\mathrm{corr}}_{\mathrm{exp},\mathrm{res}} = 0.17$). Supporting preparedness decisions thus emerge as the class most closely tied to resistance, although a larger supporting share is at the same time associated with a worse loss in $75\,\%$ of instances ($\bar\rho^{\mathrm{corr}}_{\mathrm{sup},\mathrm{loss}} = 0.60$). A larger expansive share, by contrast, is associated with a better loss in $65\,\%$ of instances ($\bar\rho^{\mathrm{corr}}_{\mathrm{exp},\mathrm{loss}} = -0.33$), while neither reinforcing preparedness decisions and loss nor supporting preparedness decisions and MPD exhibit a consistent direction ($\lvert\bar\rho^{\mathrm{corr}}\rvert \le 0.18$).

\emph{On the set of computed solutions, different preparedness instruments serve different resilience criteria: expansive decisions serve in median rebound, reinforcing decisions serve in median MPD, and supporting decisions serve in median the resistance phase. Loss is the exception, as it is tied to no single class but associated with them in opposing directions, so that it cannot be steered through the budget composition in one consistent direction. More precisely, budget compositions that prioritize rebound, resistance, or MPD can differ substantially but improve loss to a similar extent. Hence, composing the preparedness budget with regard to loss alone leaves the genuine conflicts among these criteria unaddressed. As a consequence, the choice of budget composition across the three preparedness classes drives, in particular, the trade-offs among rebound, resistance, and MPD. In contrast, setting a target value for loss only leaves considerable flexibility in budget composition.}
 

\medskip

Next, we analyze the point in time at which supporting capacities are deployed and relate it to the criteria of resilience in time, rebound and resistance. To this end we consider, for a solution~$z^*$, the \emph{resistance phase} $\mathcal{T}^{\mathrm{resist}}(z^*)$, the time window from the event until the first performance drop; the \emph{absorption phase} $\mathcal{T}^{\mathrm{absorpt}}(z^*)$, from the first drop to the time step at which the maximum performance drop is attained; and the \emph{recovery phase} $\mathcal{T}^{\mathrm{recov}}(z^*)$, from the maximum drop to full recovery. For each event~$E \in \mathcal{E}$ and phase~$p \in \{\mathrm{resist}, \mathrm{absorpt}, \mathrm{recov}\}$, let $\mathcal{T}^{p,E}(z^*)\subseteq\mathcal{T}^{E}$ collect the time steps of event~$E$ falling in phase~$p$. 
Weighting each event by $w^{E}$, let the \emph{phase share}~$b^{\mathrm{cap}}_{p}(z^*)$ of a capacity type~$\mathrm{cap} \in \{\mathrm{tr}, \mathrm{su}, \mathrm{st}\}$ be its event-weighted activation released during phase~$p$ as a fraction of its total event-weighted activation over~$\mathcal{T}^{E}(z^*)$. 
Per instance~$I \in \mathcal{I}$, we report the median~$\bar{b}^{\mathrm{cap}}_{p}(I)$ of the phase shares of each capacity type~$\mathrm{cap} \in \{\mathrm{tr}, \mathrm{su}, \mathrm{st}\}$ over the computed set of solutions~$Z^{*}_I$.
Based on that, for each instance~$I \in \mathcal{I}$, capacity type~$\mathrm{cap} \in \{\mathrm{tr}, \mathrm{su}, \mathrm{st}\}$, phase~$p \in \{\mathrm{resist}, \mathrm{absorpt}, \mathrm{recov}\}$, and criterion~$i \in \{\mathrm{reb}, \mathrm{res}\}$, we compute the Spearman rank correlation $\rho^{\mathrm{corr}}_{p,i}(I)$ between the phase shares $\bigl(b^{\mathrm{cap}}_{p}(z^*)\bigr)_{z^* \in Z^*_I}$ and the criterion values $\bigl(F^{i}(z^*)\bigr)_{z^* \in Z^*_I}$. As before, all criteria are minimized, so a negative coefficient indicates a longer resistance phase or a shorter rebound duration. Across instances, we summarize each by the median~$\bar\rho^{\mathrm{corr}}_{p,i}$.
\begin{figure}[tb]
    \centering
        \begin{subfigure}[t]{0.49\textwidth}
        \begin{tikzpicture}
        \begin{axis}[
            boxplot/draw direction=y,
            ylabel={$\bar{b}^{\mathrm{cap}}_{p}(I)$},
            ymin=-0.05, ymax=1.05, ytick={0,0.2,0.4,0.6,0.8,1},
            xmin=0.3, xmax=11.7, xtick={2,6,10},
            xticklabels={Res.\ Phase, Abs.\ Phase, Rec.\ Phase},
            xtick style={draw=none},
            xlabel = {\phantom{Phase}},
            font=\footnotesize, height=11cm, width=12cm,
            every axis plot/.append style={black, boxplot/box extend=0.7},
            cycle list={{}},
            legend cell align=left, legend pos=north east,
            legend style={font=\footnotesize,draw=none,fill=none},
            scale=0.55, transform shape
            ]
            \addplot+[forget plot, fill=black!45, boxplot prepared={draw position=1, lower whisker=0.122, lower quartile=0.370, median=0.484, upper quartile=0.641, upper whisker=0.803}] coordinates {};
            \addplot+[forget plot, fill=black!25, boxplot prepared={draw position=2, lower whisker=0.012, lower quartile=0.106, median=0.182, upper quartile=0.290, upper whisker=0.428}] table[row sep=\\,y index=0]{0.624\\};
            \addplot+[forget plot, fill=black!10, boxplot prepared={draw position=3, lower whisker=0.137, lower quartile=0.439, median=0.576, upper quartile=0.690, upper whisker=1.000}] coordinates {};
            \addplot+[forget plot, fill=black!45, boxplot prepared={draw position=5, lower whisker=0.086, lower quartile=0.281, median=0.375, upper quartile=0.473, upper whisker=0.631}] coordinates {};
            \addplot+[forget plot, fill=black!25, boxplot prepared={draw position=6, lower whisker=0.250, lower quartile=0.437, median=0.575, upper quartile=0.725, upper whisker=0.841}] coordinates {};
            \addplot+[forget plot, fill=black!10, boxplot prepared={draw position=7, lower whisker=0.000, lower quartile=0.214, median=0.323, upper quartile=0.423, upper whisker=0.682}] table[row sep=\\,y index=0]{0.749\\};
            \addplot+[forget plot, fill=black!45, boxplot prepared={draw position=9, lower whisker=0.000, lower quartile=0.063, median=0.110, upper quartile=0.175, upper whisker=0.337}] table[row sep=\\,y index=0]{0.355\\};
            \addplot+[forget plot, fill=black!25, boxplot prepared={draw position=10, lower whisker=0.000, lower quartile=0.053, median=0.132, upper quartile=0.306, upper whisker=0.583}] coordinates {};
            \addplot+[forget plot, fill=black!10, boxplot prepared={draw position=11, lower whisker=0.000, lower quartile=0.014, median=0.086, upper quartile=0.147, upper whisker=0.309}] table[row sep=\\,y index=0]{0.403\\ 0.448\\};
        \end{axis}
        \end{tikzpicture}
    \end{subfigure}
    \begin{subfigure}[t]{0.49\textwidth}
    \begin{tikzpicture}
\begin{axis}[
    boxplot/draw direction=x, y dir=reverse,
    xlabel={$\rho^{\mathrm{corr}}$},
    xmin=-1.18, xmax=1.18, xtick={-1,-0.5,0,0.5,1},
    xticklabels={$-1$,$-0.5$,$0$,$0.5$,$1$},
    ymin=0.3, ymax=23.7, ytick={2,6,10,14,18,22},
    yticklabels={$\rhocorract_{\mathrm{resist},\mathrm{reb}}$, $\rhocorract_{\mathrm{resist},\mathrm{res}}$, $\rhocorract_{\mathrm{absorpt},\mathrm{reb}}$, $\rhocorract_{\mathrm{absorpt},\mathrm{res}}$, $\rhocorract_{\mathrm{recov},\mathrm{reb}}$, $\rhocorract_{\mathrm{recov},\mathrm{res}}$},
    ytick style={draw=none},
    font=\footnotesize, height=11cm, width=12cm,
    every axis plot/.append style={black, boxplot/box extend=0.7},
    cycle list={{}},
    legend cell align=left, legend columns=3,
    legend style={font=\footnotesize,draw=none,fill=none,
        at={(0.5,1.02)},anchor=south,/tikz/every even column/.append style={column sep=8pt}},
    scale=0.55, transform shape
    ]
    \draw[dashed,gray] (axis cs:0,0.3) -- (axis cs:0,23.7);
    \addplot+[forget plot, fill=black!45, boxplot prepared={draw position=1, lower whisker=-1.000, lower quartile=-0.635, median=-0.220, upper quartile=0.208, upper whisker=1.000}] coordinates {};
    \addplot+[forget plot, fill=black!25, boxplot prepared={draw position=2, lower whisker=-1.000, lower quartile=-0.474, median=0.633, upper quartile=0.968, upper whisker=1.000}] coordinates {};
    \addplot+[forget plot, fill=black!10, boxplot prepared={draw position=3, lower whisker=-1.000, lower quartile=-0.069, median=0.357, upper quartile=0.584, upper whisker=1.000}] coordinates {};
    \addplot+[forget plot, fill=black!45, boxplot prepared={draw position=5, lower whisker=-1.000, lower quartile=-0.720, median=-0.399, upper quartile=-0.036, upper whisker=0.949}] coordinates {};
    \addplot+[forget plot, fill=black!25, boxplot prepared={draw position=6, lower whisker=-1.000, lower quartile=-1.000, median=-0.879, upper quartile=-0.200, upper whisker=0.667}] table[row sep=\\,y index=0]{1.000\\};
    \addplot+[forget plot, fill=black!10, boxplot prepared={draw position=7, lower whisker=-1.000, lower quartile=-0.651, median=-0.321, upper quartile=0.002, upper whisker=0.866}] table[row sep=\\,y index=0]{1.000\\};
    \addplot+[forget plot, fill=black!45, boxplot prepared={draw position=9, lower whisker=-1.000, lower quartile=-0.126, median=0.300, upper quartile=0.573, upper whisker=1.000}] coordinates {};
    \addplot+[forget plot, fill=black!25, boxplot prepared={draw position=10, lower whisker=-1.000, lower quartile=-0.405, median=0.457, upper quartile=0.909, upper whisker=1.000}] coordinates {};
    \addplot+[forget plot, fill=black!10, boxplot prepared={draw position=11, lower whisker=-1.000, lower quartile=-0.450, median=-0.200, upper quartile=0.107, upper whisker=0.717}] table[row sep=\\,y index=0]{1.000\\};
    \addplot+[forget plot, fill=black!45, boxplot prepared={draw position=13, lower whisker=-0.496, lower quartile=0.122, median=0.392, upper quartile=0.681, upper whisker=1.000}] coordinates {};
    \addplot+[forget plot, fill=black!25, boxplot prepared={draw position=14, lower whisker=-1.000, lower quartile=-0.400, median=0.435, upper quartile=1.000, upper whisker=1.000}] coordinates {};
    \addplot+[forget plot, fill=black!10, boxplot prepared={draw position=15, lower whisker=-0.559, lower quartile=0.085, median=0.405, upper quartile=0.657, upper whisker=1.000}] table[row sep=\\,y index=0]{-1.000\\ -0.778\\};
    \addplot+[forget plot, fill=black!45, boxplot prepared={draw position=17, lower whisker=-1.000, lower quartile=-0.259, median=0.123, upper quartile=0.432, upper whisker=1.000}] coordinates {};
    \addplot+[forget plot, fill=black!25, boxplot prepared={draw position=18, lower whisker=-1.000, lower quartile=-0.753, median=-0.500, upper quartile=-0.065, upper whisker=0.500}] table[row sep=\\,y index=0]{1.000\\};
    \addplot+[forget plot, fill=black!10, boxplot prepared={draw position=19, lower whisker=-1.000, lower quartile=-0.816, median=-0.231, upper quartile=0.022, upper whisker=1.000}] coordinates {};
    \addplot+[forget plot, fill=black!45, boxplot prepared={draw position=21, lower whisker=-1.000, lower quartile=-0.400, median=-0.128, upper quartile=0.219, upper whisker=1.000}] coordinates {};
    \addplot+[forget plot, fill=black!25, boxplot prepared={draw position=22, lower whisker=-1.000, lower quartile=-0.950, median=-0.258, upper quartile=0.634, upper whisker=1.000}] coordinates {};
    \addplot+[forget plot, fill=black!10, boxplot prepared={draw position=23, lower whisker=-1.000, lower quartile=-0.290, median=-0.057, upper quartile=0.227, upper whisker=1.000}] coordinates {};
    \end{axis}
    \end{tikzpicture}
       \end{subfigure}
    \caption{ 
    (Left)~Distribution across instances of the per-instance median phase shares $\bar b^{\mathrm{cap}}_{p}$, $\mathrm{cap} \in \{\mathrm{tr}, \mathrm{su}, \mathrm{st}\}$, i.e.\ the fraction of each capacity type's deployment released during the resistance, absorption, and recovery phase. The three phase shares of a given capacity type sum to one for each solution; the reported medians need not.
    (Right)~Distribution across instances of the per-instance Spearman rank correlation $\rhocorract_{p,i}$ between the phase share of each capacity type in phase and the resilience-in-time criteria $i\in\{\mathrm{reb}, \mathrm{res}\}$. The box shading identifies the capacity type: \swatch{black!45}~transit, \swatch{black!25}~supply, \swatch{black!10}~storage.}
    \label{fig:supporting-phase-shares}
    \label{fig:supporting-phase-corr}
\end{figure}
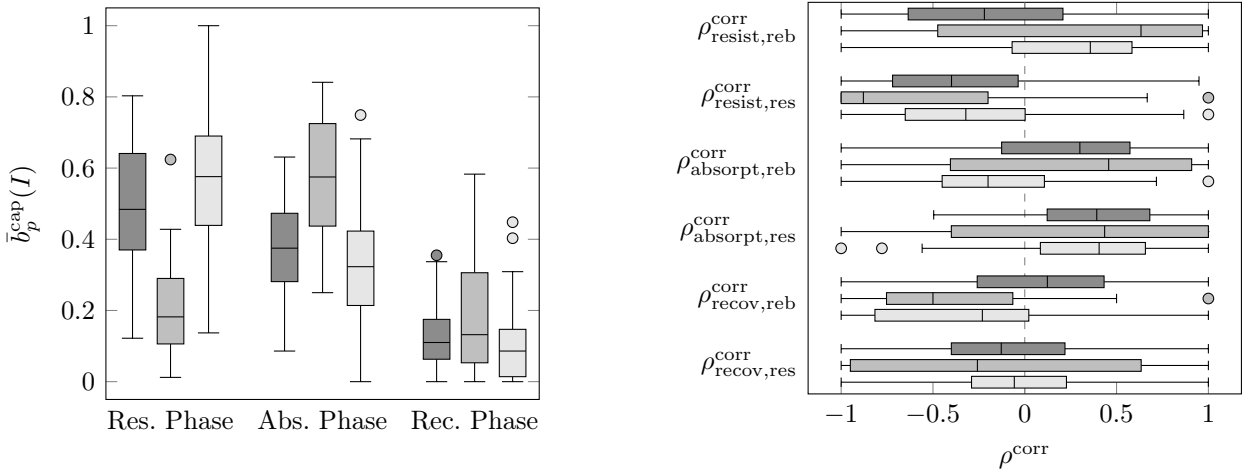

Figure~\ref{fig:supporting-phase-shares} (Left) reports the distribution of $\bar{b}^{\mathrm{cap}}_{p}(I)$, $I \in \mathcal{I}$. Transit and storage capacities are deployed predominantly during the resistance phase (median $\bar{b}^{\mathrm{tr}}_{\mathrm{resist}} = 0.48$ and $\bar{b}^{\mathrm{st}}_{\mathrm{resist}} = 0.57$), whereas supply capacities reverse this ordering and are deployed mainly during the absorption phase ($\bar{b}^{\mathrm{su}}_{\mathrm{absorpt}} = 0.57$, against $\bar{b}^{\mathrm{su}}_{\mathrm{resist}} = 0.18$ in the resistance phase). Supporting investment is thus not a single instrument released at one point in time: transit and storage capacities act early, supply capacities later.

Figure~\ref{fig:supporting-phase-corr} (Right) reports the distribution of the Spearman rank correlations $\rho^{\mathrm{corr}}_{p,i}(I)$, $I \in \mathcal{I}$. For the resistance criterion, the association is directionally robust: for every capacity type a larger share deployed during the resistance phase is associated with a longer resistance phase ($\bar\rho^{\mathrm{corr}}_{\mathrm{resist},\mathrm{res}} = -0.88$ for supply, $\bar\rho^{\mathrm{corr}}_{\mathrm{resist},\mathrm{res}} = -0.40$ for transit, $\bar\rho^{\mathrm{corr}}_{\mathrm{resist},\mathrm{res}} = -0.32$ for storage), and the corresponding interquartile ranges lie almost entirely below zero, so the direction holds on the large majority of instances rather than only at the median. For the rebound criterion the relevant phase differs by capacity type. For transit capacities an early deployment also shortens the rebound duration ($\bar\rho^{\mathrm{corr}}_{\mathrm{resisst},\mathrm{reb}} = -0.22$), improving both criteria at once; for supply and storage capacities a larger resistance phase share instead lengthens the rebound duration ($\bar\rho^{\mathrm{corr}}_{\mathrm{resist},\mathrm{reb}} = 0.63$ and $\bar\rho^{\mathrm{corr}}_{\mathrm{resist},\mathrm{reb}} = 0.36$, respectively), and it is a deployment during the recovery phase that shortens it ($\bar\rho^{\mathrm{corr}}_{\mathrm{recov},\mathrm{reb}} = -0.50$ and $\bar\rho^{\mathrm{corr}}_{\mathrm{recov},\mathrm{reb}} = -0.23$, respectively). These rebound associations are markedly less consistent than the resistance ones: their interquartile ranges typically straddle zero, so that even where the median points to a lengthening or a shortening, the sign is not constant across instances and no timing policy improves the rebound duration uniformly. The timing of supporting deployment is thus subject to the same instance-dependence found for the marginal trade-off rates. In particular, for supply and storage capacities the displacement toward the resistance phase that lengthens the resistance phase also tends to lengthen the rebound duration; but the extent of this effect, and even the instances on which its direction holds, cannot be fixed a priori.

\emph{The conflict between resistance and rebound reappears as a question of when supporting capacities are activated. 
For transit reserves, the two criteria do not conflict in timing, since an early release tends to lengthen the resistance phase and shorten the rebound duration alike. For supply and storage reserves, however, no single release policy serves both criteria, as the early release tends to lengthen the resistance phase and lengthen the rebound duration. The timing of supply and storage activation is therefore itself a decision over which the two resilience criteria in time trade off. Consequently, a preparedness plan that fixes only how much supporting capacity to invest in, and not when to release it, leaves this conflict unresolved.
}




\section{Conclusion}\label{sec:conclusion}

In this work, we proposed a unified set of resilience criteria encompassing rebound, resistance, loss, and maximum performance degradation (MPD), alongside the complementary conditions of preparedness and effectiveness. We integrated these criteria into a two-stage multicriteria optimization formulation on flows over time, allowing the trade-offs between potentially conflicting resilience objectives to be examined explicitly. We established a decomposition into a family of bicriteria linear programs and showed that the number of efficient temporal resilience profiles can grow superpolynomially in the encoding length, precluding a polynomial-size enumeration of the complete nondominated set in general. Based on this decomposition, we developed an enclosure-based heuristic for approximating the nondominated set.

On the benchmark instances studied and across the computed sets of solutions, rebound varies more substantially than resistance, loss, and MPD, giving recovery-time targets the strongest discriminatory power among candidate solutions. Moreover, the magnitude and, in some cases, the direction of the trade-offs vary across instances. This supports an explicit, instance-specific multicriteria analysis rather than representing resilience through a single universal scalar. The conflicts also affect preparedness decisions: expansive decisions are primarily associated with improved rebound, reinforcing decisions with improved MPD, and supporting decisions with an extended resistance phase. Loss is less strongly associated with any single decision class. Finally, the conflict between resistance and rebound concerns not only how much supporting capacity is prepared, but also when it is activated. Preparedness strategies should therefore specify both the composition and the timing of capacity deployment.

Taken together, the results show that resilience can comprise several potentially contradictory goals whose relevance and trade-offs must be assessed for the system at hand. The computational findings are based on synthetic, benchmark-derived instances and on the solution sets returned by the first two phases of the proposed heuristic. Evaluating the complete procedure, comparing its results with exact nondominated sets on smaller instances, and validating the identified relationships in empirical applications therefore constitute important directions for future research. Further extensions may incorporate perishable demand, cascading impacts, cost-based and multi-commodity flows, and interdependent networks. The framework could also be used to investigate the price of resilience, in line with related work on the price of robustness~\citep{bertsimas_price_2004}.

\section*{Code and Data Disclosure}\label{sec:Code and Data Disclosure}

\noindent
The code and data to support the numerical experiments in this paper can be found at \url{https://gitlab.kit.edu/stephan.helfrich/resilient-network-flows} 

\section*{Declaration of Interests}

\noindent
The authors declare that they have no known competing financial interests or personal relationships that could have appeared to influence the work reported in this paper.

\section*{Declaration of Generative AI in Scientific Writing}

\noindent
During the preparation of this work, the author(s) used Claude in order to improve  readability and grammar. After using this tool, the author(s) reviewed and edited the content as needed and take(s) full responsibility for the content of the published article.

\section*{Author Contributions}

\vspace{-6pt}
\noindent
\emph{Stephan Helfrich}: Conceptualization, Methodology, Software, Data Curation, Writing - Original Draft, Writing - Review \& Editing, Visualization\\
\emph{Gabriela Ciolacu}: Conceptualization, Methodology, Data Curation, Writing - Original Draft, Writing - Review \& Editing, Visualization\\
\emph{Jan Boeckmann}: Conceptualization, Methodology, Software, Writing - Original Draft, Writing - Review \& Editing, Visualization\\
\emph{Emilia Grass}: Project Administration, Funding Acquisition, Writing - Review \& Editing

\section*{Acknowledgments}

\vspace{-6pt}
\noindent
The authors thank Eric Alexander Hoffeins and Manuel Mechnich for their assistance in implementing and carrying out the numerical study.
Further, the authors acknowledge support by the state of Baden-W\"urttemberg through bwHPC.

\bibliography{references}

\appendix
\section{Complete Model Formulation}\label{app:complete-model-formulation}
We collect the full statement of the two-stage multicriteria optimization problem~\eqref{eq:two-stage-problem} in the notation of Sections~\ref{sec:problem-formulation}--\ref{sec:objectives}. The system is a directed network $G=(V,A)$ over the discrete time horizon $\mathcal{T}=\{0,1,\ldots,T\}$, subject to a finite set of adverse events $\mathcal{E}$; each event $E\in\mathcal{E}$ has a known start time $\theta^E\in\mathcal{T}$, duration $\ell^E$, weight $w^E$, and residual horizon $\mathcal{T}^E=\{\theta^E,\ldots,T\}$. To guarantee feasibility of the reactive flow, the vertex set is augmented by an artificial disposal vertex $v^{\mathrm{disp}}$ of infinite storage capacity, and every vertex is connected to it by an uncapacitated, zero-transit-time arc; flow routed to $v^{\mathrm{disp}}$ represents discarded flow. The nominal and reactive excess $\textup{ex}_v(\theta)$ and $\textup{ex}^E_v(\theta)$ are as defined in Section~\ref{sec:flow-over-time}. The parameters and decision variables are summarized in Tables~\ref{app:tab:params} and~\ref{app:tab:vars}.

\begin{small}
\begin{longtable}{@{}p{0.24\linewidth} p{0.70\linewidth}@{}}
\caption{Parameters of the formulation.}
\label{app:tab:params}\\
\toprule
Symbol & Description \\
\midrule
\endfirsthead
\toprule
Symbol & Description \\
\midrule
\endhead
\bottomrule
\endfoot
$T$ & Length of the time horizon $\mathcal{T}=\{0,\ldots,T\}$. \\
$\tau_a\in\NN$ & Transit time of arc $a\in A$. \\
$\bar f_a(\theta)$ & Nominal transit capacity of arc $a$ at time $\theta$. \\
$\bar s_v(\theta)$ & Nominal supply capacity at vertex $v$. \\
$\bar d_v(\theta)$ & Nominal demand at vertex $v$. \\
$\bar\kappa_v(\theta)$ & Nominal storage capacity at vertex $v$. \\
$\bar f^{\mathrm{exp}}_a,\ \bar s^{\mathrm{exp}}_v,\ \bar\kappa^{\mathrm{exp}}_v$ & Upper bounds on expansive arc / supply / storage expansion. \\
$\bar f^{\mathrm{sup}}_a,\ \bar s^{\mathrm{sup}}_v,\ \bar\kappa^{\mathrm{sup}}_v$ & Upper bounds on supporting arc / supply / storage reserves. \\
$\Delta^E\bar f_a(\theta)\le 0$ & Transit-capacity reduction under event $E$. \\
$\Delta^E\bar s_v(\theta)\le 0$ & Supply reduction under event $E$. \\
$\Delta^E\bar d_v(\theta)\ge 0$ & Demand surge under event $E$. \\
$\Delta^E\bar\kappa_v(\theta)\le 0$ & Storage-capacity reduction under event $E$. \\
$\theta^E,\ \ell^E$ & Start time and duration of event $E$. \\
$w^E\ge 0$ & Normalized weight of event $E$ ($\sum_{E\in\mathcal{E}} w^E=1$). \\
$c_\alpha,\ c_\beta,\ c_\gamma$ & Cost coefficients of expansive arc / supply / storage decisions. \\
$c_{\bar f},\ c_{\bar s},\ c_{\bar\kappa}$ & Cost coefficients of reinforcing arc / supply / storage decisions. \\
$c_\Phi,\ c_\Pi,\ c_\Psi$ & Cost coefficients of supporting arc / supply / storage budgets. \\
$B^{\mathrm{prep}}$ & Preparedness budget. \\
$B^{\mathrm{prep,\,str}}$ & Bound(s) of the structural preparedness restrictions $H^{\mathrm{prep}}$. \\
$x^{\mathrm{ref}}$ & Reference nominal flow over time. \\
$B^{\mathrm{effect}}$ & Effectiveness tolerance(s) for $H^{\mathrm{effect}}$. \\
\end{longtable}
\begin{longtable}{@{}p{0.20\linewidth} p{0.22\linewidth} p{0.50\linewidth}@{}}
\caption{Decision variables of the formulation.}
\label{app:tab:vars}\\
\toprule
Symbol & Domain & Description \\
\midrule
\endfirsthead
\toprule
Symbol & Domain & Description \\
\midrule
\endhead
\bottomrule
\endfoot
\multicolumn{3}{@{}l}{\emph{First stage --- nominal flow over time}}\\
$f_a(\theta)$ & $\ge 0$ & Nominal flow entering arc $a$ at time $\theta$. \\
$s_v(\theta)$ & $\ge 0$ & Nominal supply generated at vertex $v$. \\
$d_v(\theta)$ & $\ge 0$ & Nominal met demand at vertex $v$. \\
$u_v(\theta)$ & free & Nominal unmet demand at vertex $v$. \\
\midrule
\multicolumn{3}{@{}l}{\emph{First stage --- preparedness}}\\
$\alpha_a$ & $[0,1]$ & Expansive arc-capacity decision. \\
$\beta_v$ & $[0,1]$ & Expansive supply-capacity decision. \\
$\gamma_v$ & $[0,1]$ & Expansive storage-capacity decision. \\
$\zeta_a$ & $[0,1]$ & Reinforcing arc decision. \\
$\sigma_v$ & $[0,1]$ & Reinforcing supply decision. \\
$\eta_v$ & $[0,1]$ & Reinforcing storage decision. \\
$\Phi,\ \Pi,\ \Psi$ & $\ge 0$ & Supporting arc / supply / storage budgets. \\
\midrule
\multicolumn{3}{@{}l}{\emph{Second stage --- reactive flow (per event $E$, $\theta\in\mathcal{T}^E$)}}\\
$f^E_a(\theta)$ & $\ge 0$ & Reactive flow entering arc $a$. \\
$s^E_v(\theta)$ & $\ge 0$ & Reactive supply generated at vertex $v$. \\
$d^E_v(\theta)$ & $\ge 0$ & Reactive met demand at vertex $v$. \\
$u^E_v(\theta)$ & free & Reactive unmet demand at vertex $v$. \\
$\phi^E_a(\theta)$ & $\ge 0$ & Activation of the supporting arc reserve. \\
$\pi^E_v(\theta)$ & $\ge 0$ & Activation of the supporting supply reserve. \\
$\psi^E_v(\theta)$ & $\ge 0$ & Activation of the supporting storage reserve. \\
\midrule
\multicolumn{3}{@{}l}{\emph{Objective auxiliaries (per event $E$, $\theta\in\mathcal{T}^E$)}}\\
$r^{\mathrm{reb},E}(\theta)$ & $\{0,1\}$ & Rebound indicator. \\
$r^{\mathrm{res},E}(\theta)$ & $\{0,1\}$ & Resistance indicator. \\
$r^{\mathrm{loss},E}(\theta)$ & $\ge 0$ & Loss variable. \\
$r^{\mathrm{MPD},E}$ & $\ge 0$ & MPD variable. \\
\end{longtable}
\end{small}
\subsection{First-Stage Decision Space}\label{app:first-stage}
\begin{align}
  f_a(\theta) &\le \bar f_a(\theta) + \bar f^{\mathrm{exp}}_a\,\alpha_a, & a&\in A,\ \theta\in\mathcal{T}, \label{app:eq:f-cap}\\
  s_v(\theta) &\le \bar s_v(\theta) + \bar s^{\mathrm{exp}}_v\,\beta_v, & v&\in V,\ \theta\in\mathcal{T}, \label{app:eq:s-cap}\\
  \textstyle\sum_{\xi=0}^{\theta} d_v(\xi) &\le \textstyle\sum_{\xi=0}^{\theta} \bar d_v(\xi), & v&\in V,\ \theta\in\mathcal{T}, \label{app:eq:d-ub}\\
  \textstyle\sum_{\xi=0}^{\theta} d_v(\xi) + \sum_{\xi=0}^{\theta} u_v(\xi) &= \textstyle\sum_{\xi=0}^{\theta} \bar d_v(\xi), & v&\in V,\ \theta\in\mathcal{T}, \label{app:eq:d-bal}\\
  0 \le \textup{ex}_v(\theta) &\le \bar\kappa_v(\theta) + \bar\kappa^{\mathrm{exp}}_v\,\gamma_v, & v&\in V,\ \theta\in\mathcal{T}, \label{app:eq:ex-cap}
\end{align}
\begin{align}
  &c_\alpha^{\top}\alpha + c_\beta^{\top}\beta + c_\gamma^{\top}\gamma + c_{\bar f}^{\top}\zeta 
  + c_{\bar s}^{\top}\sigma + c_{\bar\kappa}^{\top}\eta + c_\Phi\Phi + c_\Pi\Pi + c_\Psi\Psi \le B^{\mathrm{prep}}, \label{app:eq:budget}\\
  &H^{\mathrm{effect}}(x,x^{\mathrm{ref}}) \le B^{\mathrm{effect}}, \label{app:eq:eff}
\end{align}
 
\subsection{Second-Stage Decision Spaces}\label{app:second-stage}
\begin{align}
  f^E_a(\theta) &\le \bigl(\bar f_a(\theta) + \bar f^{\mathrm{exp}}_a\,\alpha_a\bigr) + (1-\zeta_a)\,\Delta^E\bar f_a(\theta) + \phi^E_a(\theta),\qquad a\in A,\ \theta\in\mathcal{T}^E, \label{app:eq:fE-cap}\\
  s^E_v(\theta) &\le \bigl(\bar s_v(\theta) + \bar s^{\mathrm{exp}}_v\,\beta_v\bigr) + (1-\sigma_v)\,\Delta^E\bar s_v(\theta) + \pi^E_v(\theta), \qquad v\in V,\ \theta\in\mathcal{T}^E, \label{app:eq:sE-cap}\\
  0 \le \textup{ex}^E_v(\theta) &\le \bigl(\bar\kappa_v(\theta) + \bar\kappa^{\mathrm{exp}}_v\,\gamma_v\bigr) + (1-\eta_v)\,\Delta^E\bar\kappa_v(\theta) + \psi^E_v(\theta),  \qquad  v\in V,\ \theta\in\mathcal{T}^E, \label{app:eq:exE-cap}
\end{align}
\begin{align}
  &\textstyle\sum_{\xi=0}^{\theta^E-1} d_v(\xi) + \sum_{\xi=\theta^E}^{\theta} d^E_v(\xi) \\
  &\le \textstyle\sum_{\xi=0}^{\theta}\bigl(\bar d_v(\xi) + \Delta^E\bar d_v(\xi)\bigr), \qquad  v\in V,\ \theta\in\mathcal{T}^E, \label{app:eq:dE-ub}\\
  \textstyle\sum_{\xi=0}^{\theta^E-1}&\bigl(d_v(\xi)+u_v(\xi)\bigr) + \textstyle\sum_{\xi=\theta^E}^{\theta}\bigl(d^E_v(\xi)+u^E_v(\xi)\bigr) \notag\\
    &= \textstyle\sum_{\xi=0}^{\theta}\bigl(\bar d_v(\xi) + \Delta^E\bar d_v(\xi)\bigr), \qquad v\in V,\ \theta\in\mathcal{T}^E, \label{app:eq:dE-bal}
\end{align}
\begin{align}
  \phi^E_a(\theta) &\le \bar f^{\mathrm{sup}}_a, & a&\in A,\ \theta\in\mathcal{T}^E,\ E\in\mathcal{E}, \label{app:eq:phi-cap}\\
  \pi^E_v(\theta)  &\le \bar s^{\mathrm{sup}}_v, & v&\in V,\ \theta\in\mathcal{T}^E,\ E\in\mathcal{E}, \label{app:eq:pi-cap}\\
  \psi^E_v(\theta) &\le \bar\kappa^{\mathrm{sup}}_v, & v&\in V,\ \theta\in\mathcal{T}^E,\ E\in\mathcal{E}, \label{app:eq:psi-cap}\\
  \textstyle\sum_{a\in A}\sum_{\theta\in\mathcal{T}^E}\phi^E_a(\theta) &\le \Phi, & E&\in\mathcal{E}, \\
  \textstyle\sum_{v\in V}\sum_{\theta\in\mathcal{T}^E}\pi^E_v(\theta) &\le \Pi, & E&\in\mathcal{E}, \\
  \textstyle\sum_{v\in V}\sum_{\theta\in\mathcal{T}^E}\psi^E_v(\theta) &\le \Psi, & E&\in\mathcal{E}. \label{app:eq:sup-budget}
\end{align}
\begin{align}
  H^{\mathrm{prep}}\bigl(\cdot \bigr) &\le B^{\mathrm{prep,\,str}}, \qquad E\in\mathcal{E}, \label{app:eq:coupling}
\end{align}
and the variable domains of Table~\ref{app:tab:vars}.
 
\subsection{Objective Functions}\label{app:objectives}
All objectives are expressed through the performance gap $g_E(\theta)=\sum_{v\in V}\sum_{\xi=\theta^E}^{\theta}\bigl(u^E_v(\xi)-u_v(\xi)\bigr)$, $\theta\in\mathcal{T}^E$. 
\begin{align}
  r^{\mathrm{reb},E}(\theta)=1 &\Rightarrow g_E(\theta)\le 0, & \theta&\in\mathcal{T}^E, \label{app:eq:reb-ind}\\
  r^{\mathrm{reb},E}(\theta) &\ge r^{\mathrm{reb},E}(\theta-1), & \theta&\in\mathcal{T}^E,\ \theta>\theta^E, \label{app:eq:reb-mono}\\
  r^{\mathrm{res},E}(\theta)=1 &\Rightarrow g_E(\theta)\le 0, & \theta&\in\mathcal{T}^E, \label{app:eq:res-ind}\\
  r^{\mathrm{res},E}(\theta) &\le r^{\mathrm{res},E}(\theta-1), & \theta&\in\mathcal{T}^E,\ \theta>\theta^E, \label{app:eq:res-mono}\\
  g_E(\theta) &\le r^{\mathrm{loss},E}(\theta), & \theta&\in\mathcal{T}^E, \label{app:eq:loss-cap}\\
  g_E(\theta) &\le r^{\mathrm{MPD},E}, & \theta&\in\mathcal{T}^E, \label{app:eq:mpd-peak}
\end{align}
\begin{align}
  \Freb(z^1,z^2) &= \textstyle\sum_{E\in\mathcal{E}} w^E\sum_{\theta\in\mathcal{T}^E}\bigl(1-r^{\mathrm{reb},E}(\theta)\bigr), \label{app:eq:obj-reb}\\
  \Fres(z^1,z^2) &= T - \textstyle\sum_{E\in\mathcal{E}} w^E\sum_{\theta\in\mathcal{T}^E} r^{\mathrm{res},E}(\theta), \label{app:eq:obj-res}\\
  \Floss(z^1,z^2) &= \textstyle\sum_{E\in\mathcal{E}} w^E\sum_{\theta=\theta^E}^{T} r^{\mathrm{loss},E}(\theta), \label{app:eq:obj-loss}\\
  \Fmpd(z^1,z^2) &= \textstyle\sum_{E\in\mathcal{E}} w^E\, r^{\mathrm{MPD},E}. \label{app:eq:obj-mpd}
\end{align}

\section{Technical Proofs}\label{app:proofs}
In the following, we present the technical proofs of Section~\ref{sec:multiobjective}. To this end, for $u,v\in\RR^k$ we write $u\leqq v$ if $u_i\le v_i$ for all $i$, and $u\le v$ if $u\leqq v$ and $u\neq v$. For a solution $z\in Z$ and an adverse event $E\in\mathcal E$, we denote by $\ell(z)=(\boldsymbol\tau^{\mathrm{res}}(z),\boldsymbol\tau^{\mathrm{reb}}(z))$ the \emph{realized temporal resilience profile} of $z$, given by $\tau^{\mathrm{res}}_E(z)=\sum_{\theta\in\mathcal T^E}r^{\mathrm{res},E}(\theta)$ and $\tau^{\mathrm{reb}}_E(z)=\sum_{\theta\in\mathcal T^E}\bigl(1-r^{\mathrm{reb},E}(\theta)\bigr)$. By the definition of the resistance and rebound indicators, $z\in\mathcal Z(\ell(z))$, and since the two temporal resilience criteria are constant on $\mathcal Z(\ell)$, we have $\Freb(z)=\Freb(\ell(z))$ and $\Fres(z)=\Fres(\ell(z))$.
 
We further use that, for each $\ell\in\mathcal L$, the residual problem~\eqref{eq:residual-bi-obj} is a bounded bicriteria linear program, so its nondominated set $Y_N(\ell)$ is \emph{externally stable}: every feasible residual image is component-wise bounded below by some element of $Y_N(\ell)$, that is, for every $z\in\mathcal Z(\ell)$ there is $y\in Y_N(\ell)$ with $y\leqq\bigl(\Floss(z),\Fmpd(z)\bigr)$~\citep{matthias_ehrgott_multicriteria_2005}. The nondominated set $Y_N$ of the two-stage problem~\eqref{eq:two-stage-problem} is externally stable in the same sense.

\subsection{Proof of Theorem~\ref{thm:decomp}}
Let $y\in Y_N$ and $z\in Z$ with $F(z)=y$. Set $\tau^{\mathrm{res}}=\Fres(z)$, $\tau^{\mathrm{reb}}=\Freb(z)$, and $\ell=(\tau^{\mathrm{res}},\tau^{\mathrm{reb}})$. Then $z\in\mathcal Z(\ell)$, and hence $\ell\in\mathcal L$. Suppose $z$ is not efficient for~\eqref{eq:residual-bi-obj}. Then there is $z'\in\mathcal Z(\ell)$ with $\bigl(\Floss(z'),\Fmpd(z')\bigr)\le\bigl(\Floss(z),\Fmpd(z)\bigr)$. Since $\Fres(z')=\Fres(z)$ and $\Freb(z')=\Freb(z)$, this yields $F(z')\le F(z)$, contradicting efficiency of $z$. Thus, $z$ is efficient for~\eqref{eq:residual-bi-obj}, and consequently $Z^*(\ell)$ contains a solution $\hat z$ with $F(\hat z)=y$. 

By construction, $Z^*(\ell)$ contains a solution $\hat z$ with $\bigl(\Floss(\hat z),\Fmpd(\hat z)\bigr)=\bigl(\Floss(z),\Fmpd(z)\bigr)$. As $\hat z\in\mathcal Z(\ell)$, the temporal criteria coincide with those of $z$, that is, $\Freb(\hat z)=\Freb(\ell)=\Freb(z)$ and $\Fres(\hat z)=\Fres(\ell)=\Fres(z)$; consequently $F(\hat z)=F(z)=y$. Since $y\in Y_N$ is nondominated, no feasible solution dominates $\hat z$, so $\hat z$ is efficient. Thus $\bigcup_{\ell\in\mathcal L}Z^*(\ell)$ contains an efficient solution with image $y$. \Halmos

\subsection{Proof of Theorem~\ref{thm:intractable}}

We construct the instance, characterize its efficient first-stage decisions, and then count the temporal resilience profiles they realize. Unless explicitly stated otherwise, input parameters are per default set to~0. Firstly, we choose~$T = 1$, that is, $\mathcal{T} = \{0, 1\}$. Abbreviating
\begin{align}\label{eq:delta}
	\delta_k \coloneqq 2^{\,K+1-k}-1 \qquad\text{for } k \in \{1,\dots,K\},
\end{align}
we introduce a source vertex~$\rho$ with supply $\bar{s}_{\rho}(1) = 2^{K+1}$ and for each $k \in \{1, \dots, K\}$ a vertex~$v_k$ with demand $\bar{d}_{v_k}(1) = \delta_k$. The arc set~$A$ consists of arcs $a_k = (\rho, v_k)$ for $k \in \{1, \dots, K\}$ with $\bar{f}_{a_k}(1) = \delta_k$.

We construct~$K$ adverse events~$E_k, k \in \{1, \dots, K\}$. For $k \in \{1, \dots, K\}$, the event~$E_k$ blocks the arc~$a_k$ by setting $\Delta^{E_k}\bar{f}_{a_k}(1) =  -\delta_k$. We can prevent the adverse event from affecting the flow over time by the possibility of adding redundant capacity of up to $\bar{f}^{\text{exp}}_{a_k} = \delta_k$ in the first stage with the decision variable~$\alpha_{a_k} \in [0, 1]$. Apart from that, there are no further possibilities to increase resilience of the system. The weights of the adverse events are chosen as
\begin{align}\label{eq:weights}
	w^{E_k} = \frac{2^{\,k}}{2^{K+1}-2} \qquad\text{for } k \in \{1, \dots, K\},
\end{align}
so that $\sum_{k=1}^K w^{E_k} = 1$. The two parameter choices~\eqref{eq:delta} and~\eqref{eq:weights} are tied to each other by the \emph{complementarity relation}
\begin{align}\label{eq:complementary}
	w^{E_k}\bigl(1+\delta_k\bigr)
	\;=\; \frac{2^{\,k}\cdot 2^{\,K+1-k}}{2^{K+1}-2} \;=\; \Lambda ,
\end{align}
which holds for all $k \in \{1,\dots,K\}$ with a constant $\Lambda \coloneqq 2^{K+1}/(2^{K+1}-2) > 1$ not depending on~$k$. 

We assume~$K$ is even and imply the preparedness constraint $\sum_{k = 1}^K \alpha_{a_k} \leq \frac{K}{2}$, obtained from unit costs for the expansive preparedness decisions and a preparedness budget of~$\frac K2$.

\medskip

A first-stage decision is a vector $\alpha = (\alpha_{a_1},\dots,\alpha_{a_K}) \in [0,1]^K$; it is feasible if and only if $\sum_{k}\alpha_{a_k} \le \frac K2$. We stress that the expansion variables are \emph{not} binary. We write
\[
	\mathcal K(\alpha) \coloneqq \{\,k \in \{1,\dots,K\} \;:\; \alpha_{a_k}=1\,\}
\]
for the set of \emph{fully} protected indices and $\bar{\mathcal K}(\alpha)$ for its complement in $\{1,\dots,K\}$. A first-stage decision is called \emph{integral} if $\alpha \in \{0,1\}^K$, and \emph{tight} if $\sum_k \alpha_{a_k} = \frac K2$.

\begin{lemma}\label{lem:image}
	Let $\alpha$ be a feasible first-stage decision. Then every efficient solution with first-stage decision~$\alpha$ has the image
	\begin{align}\label{eq:image}
		\Fres &= \Freb = \sum_{k\in\bar{\mathcal K}(\alpha)} w^{E_k}, \\
		\Floss &= \Fmpd = \sum_{k=1}^{K} w^{E_k}\delta_k\bigl(1-\alpha_{a_k}\bigr). \nonumber
	\end{align}
	If, in addition, $\alpha$ is integral and tight, then
	\begin{align}\label{eq:line}
		\Fres + \Floss \;=\; \tfrac K2\,\Lambda .
	\end{align}
\end{lemma}

\begin{proof}{Proof.}
	The nominal flow sends $\delta_k$ units along~$a_k$ for every~$k$, which is feasible because $\bar s_\rho(1) = 2^{K+1} \ge \sum_{k}\delta_k$, and meets every demand, so $P_x\equiv 0$. Under the event~$E_k$ the impacted capacity of~$a_k$ is $\bar f_{a_k}(1) + \bar f^{\textup{exp}}_{a_k}\alpha_{a_k} + \Delta^{E_k}\bar f_{a_k}(1) = \delta_k\,\alpha_{a_k}$, while all remaining arcs keep their capacity~$\delta_j$ and hence exactly serve their own demand~$\delta_j$. Every reactive flow therefore satisfies
	\[
		g_{E_k}(0) = 0, \qquad g_{E_k}(1) \;\ge\; \delta_k\bigl(1-\alpha_{a_k}\bigr),
	\]
	with equality if and only if the impacted capacity of~$a_k$ is saturated. As all four criteria are non-decreasing in $g_{E_k}(1)$, only reactive flows attaining equality can be efficient, and we restrict to those.

	Since $\mathcal T^{E_k}=\{1\}$ and $P_x\equiv0$, we have $r^{\textup{res},E_k}(1)=1$ if and only if $g_{E_k}(1)\le 0$, that is, if and only if $\alpha_{a_k}=1$. Hence $\tau^{\textup{res}}_{E_k} = \mathbb 1_{\{\alpha_{a_k}=1\}}$ and $\tau^{\textup{reb}}_{E_k} = 1-\mathbb 1_{\{\alpha_{a_k}=1\}}$, whereas $\Rloss(x,x^{E_k}) = \Rmpd(x,x^{E_k}) = \delta_k(1-\alpha_{a_k})$ depends on~$\alpha_{a_k}$ linearly. Aggregating over the events with $T=1$ and $\sum_k w^{E_k}=1$ yields~\eqref{eq:image}.

	If $\alpha$ is integral and tight, then $\alpha_{a_k}=0$ for every $k\in\bar{\mathcal K}(\alpha)$ and $\lvert\bar{\mathcal K}(\alpha)\rvert = \frac K2$, so that summing~\eqref{eq:complementary} over $k\in\bar{\mathcal K}(\alpha)$ gives $\Fres+\Floss = \sum_{k\in\bar{\mathcal K}(\alpha)}w^{E_k}(1+\delta_k) = \frac K2\Lambda$.
\end{proof}

Note that the four criteria collapse pairwise on this instance; the statement of Theorem~\ref{thm:intractable} therefore already holds for the bicriteria subproblem consisting of one temporal and one performance criterion. Note further that partial expansion is not worthless: it does not improve the temporal criteria at all, since these react only to $\alpha_{a_k}=1$, but it does improve the performance criteria linearly. The following proposition shows that this is never enough.

\begin{proposition}\label{prop:all-efficient}
	A feasible first-stage decision is efficient if and only if it is integral and tight. Distinct such decisions have distinct images, and there are $\binom{K}{K/2}$ of them.
\end{proposition}

\begin{proof}{Proof.}
	\emph{(i) Integral tight decisions are pairwise incomparable.} Let $\alpha,\alpha'$ be integral and tight with $\alpha\neq\alpha'$, and abbreviate $\mathcal K \coloneqq \mathcal K(\alpha)$, $\mathcal K' \coloneqq \mathcal K(\alpha')$, so that $\lvert\mathcal K\rvert = \lvert\mathcal K'\rvert = \frac K2$ and $\mathcal K \neq \mathcal K'$. Since $w^{E_k}$ is a positive multiple of~$2^k$, distinct subsets of $\{1,\dots,K\}$ have distinct weight sums, so $\Fres(\alpha)\neq\Fres(\alpha')$ by~\eqref{eq:image}; in particular the images are distinct. Assume without loss of generality $\Fres(\alpha)<\Fres(\alpha')$. Both images satisfy~\eqref{eq:line}, whence
	\begin{align*}
		\Floss(\alpha) &= \tfrac K2\Lambda - \Fres(\alpha) \\
		&>\; \tfrac K2\Lambda - \Fres(\alpha') = \Floss(\alpha') ,
	\end{align*}
	so neither image dominates the other.

	\emph{(ii) Every other feasible decision is dominated.} Let $\alpha$ be feasible and not both integral and tight. In either of the two cases below we construct a feasible~$\alpha'$ with $\Fres(\alpha')\le\Fres(\alpha)$ and $\Floss(\alpha')<\Floss(\alpha)$, so that $\alpha'$ dominates~$\alpha$ by~\eqref{eq:image} and $\alpha$ is not efficient. We use twice that enlarging a coordinate of~$\alpha$ can only enlarge $\mathcal K(\alpha)$ and hence, by~\eqref{eq:image}, can only decrease~$\Fres$. Moreover, \eqref{eq:complementary} gives
	\begin{align}\label{eq:decreasing}
		w^{E_k}\delta_k = \Lambda - w^{E_k} ,
	\end{align}
	which is positive and strictly decreasing in~$k$ because $w^{E_k}$ is strictly increasing in~$k$.

	First assume that $\alpha$ is \emph{not tight}, that is, $t \coloneqq \frac K2 - \sum_{k}\alpha_{a_k} > 0$. Since $\sum_{k}\alpha_{a_k} \le \frac K2 < K$, there is an index~$k$ with $\alpha_{a_k} < 1$. Let $\alpha'$ agree with~$\alpha$ except for $\alpha'_{a_k} \coloneqq \alpha_{a_k} + \varepsilon$, where $\varepsilon \coloneqq \min\{t,\, 1-\alpha_{a_k}\} > 0$. Then $\alpha'\in[0,1]^K$ consumes at most $\frac K2$ of the budget and is thus feasible, $\Fres(\alpha')\le\Fres(\alpha)$, and
	\[
		\Floss(\alpha') = \Floss(\alpha) - \varepsilon\, w^{E_k}\delta_k \;<\; \Floss(\alpha)
	\]
	by~\eqref{eq:decreasing}.

	Now assume that $\alpha$ is tight but \emph{not integral} and put $F \coloneqq \{\,k : 0 < \alpha_{a_k} < 1\,\} \neq \emptyset$. As $\sum_{k\in F}\alpha_{a_k} = \frac K2 - \lvert\mathcal K(\alpha)\rvert$ is an integer while no single element of~$(0,1)$ is, we have $\lvert F\rvert \ge 2$. Choose $k_1,k_2 \in F$ with $k_1 < k_2$ and let $\alpha'$ agree with~$\alpha$ except for
	\[
		\alpha'_{a_{k_1}} \coloneqq \alpha_{a_{k_1}} + \varepsilon , \qquad
		\alpha'_{a_{k_2}} \coloneqq \alpha_{a_{k_2}} - \varepsilon ,
	\]
	where $\varepsilon \coloneqq \min\{1-\alpha_{a_{k_1}},\, \alpha_{a_{k_2}}\} > 0$. Then $\alpha'\in[0,1]^K$ consumes exactly the same budget as~$\alpha$ and is thus feasible. Since $\alpha'_{a_{k_2}} < 1$, the index~$k_2$ lies in neither $\mathcal K(\alpha)$ nor $\mathcal K(\alpha')$, so $\mathcal K(\alpha) \subseteq \mathcal K(\alpha')$ and $\Fres(\alpha') \le \Fres(\alpha)$. Finally, by~\eqref{eq:decreasing} and $k_1 < k_2$,
	\begin{align*}
		&\Floss(\alpha') - \Floss(\alpha) \\
		&\quad = \varepsilon\bigl(w^{E_{k_2}}\delta_{k_2} - w^{E_{k_1}}\delta_{k_1}\bigr)
		 = \varepsilon\bigl(w^{E_{k_1}} - w^{E_{k_2}}\bigr) < 0 .
	\end{align*}

	\emph{(iii) Integral tight decisions are efficient.} Let $\alpha'$ be integral and tight and suppose some feasible $\beta$ dominates it. If $\beta$ is integral and tight, this contradicts~(i). Otherwise, by~(ii) there is an integral tight $\beta'$ dominating~$\beta$; by transitivity $\beta'$ dominates~$\alpha'$, again contradicting~(i).

	Finally, integral tight decisions are exactly the incidence vectors of the $\binom{K}{K/2}$ subsets of $\{1,\dots,K\}$ of cardinality~$\frac K2$.
\end{proof}
We are now ready to prove Theorem~\ref{thm:intractable}.
\begin{proof}{Proof of Theorem~\ref{thm:intractable}.}
	Recall that a temporal resilience profile $\ell\in\mathcal L$ is \emph{efficient} if its lifted image set contains a nondominated image, that is, $\widehat Y(\ell)\cap Y_N\neq\emptyset$.

	By Proposition~\ref{prop:all-efficient} the efficient first-stage decisions are exactly the integral tight ones, that is, the incidence vectors of the subsets $\mathcal K\subseteq\{1,\dots,K\}$ with $\lvert\mathcal K\rvert=\frac K2$; there are $\binom{K}{K/2}$ of them and their images are pairwise distinct. Let $z$ be an efficient solution with such a first-stage decision~$\alpha$. Since $\mathcal T^{E_k}=\{1\}$ for every~$k$, the proof of Lemma~\ref{lem:image} gives $\tau^{\mathrm{res}}_{E_k}(z)=\mathbb 1_{\{\alpha_{a_k}=1\}}$ and $\tau^{\mathrm{reb}}_{E_k}(z)=1-\mathbb 1_{\{\alpha_{a_k}=1\}}$, so the realized temporal resilience profile satisfies $\boldsymbol\tau^{\mathrm{res}}(z)=\mathbb 1_{\mathcal K(\alpha)}$ and $\boldsymbol\tau^{\mathrm{reb}}(z)=\mathbb 1_{\bar{\mathcal K}(\alpha)}$. This profile depends only on~$\mathcal K(\alpha)$ and not on the particular efficient solution~$z$; we denote it by $\ell(\alpha)$.

	The assignment $\alpha\mapsto\ell(\alpha)$ is injective, its inverse reading~$\mathcal K$ off the support of~$\boldsymbol\tau^{\mathrm{res}}$; hence the $\binom{K}{K/2}$ integral tight decisions realize $\binom{K}{K/2}$ distinct temporal resilience profiles. Each profile~$\ell(\alpha)$ is efficient: the efficient solution~$z$ lies in $\mathcal Z(\ell(\alpha))$, so its image $F(z)$ belongs to the lifted image set $\widehat Y(\ell(\alpha))$, and $F(z)$ is nondominated by Proposition~\ref{prop:all-efficient}. Distinct integral tight decisions $\alpha\neq\alpha'$ satisfy $\bar{\mathcal K}(\alpha)\neq\bar{\mathcal K}(\alpha')$, and since each $w^{E_k}$ is a positive multiple of~$2^{k}$, distinct index sets have distinct weight sums, so $\Fres(\alpha)\neq\Fres(\alpha')$ by~\eqref{eq:image} and the two nondominated images differ. The lifted image sets of these profiles therefore each contain a distinct nondominated image.

	Conversely, every nondominated image is the image of an efficient solution, which by Proposition~\ref{prop:all-efficient} is integral tight and hence realizes one of the profiles~$\ell(\alpha)$. Thus the instance has exactly $\binom{K}{K/2}$ temporal resilience profiles that yield distinct nondominated images.

	Finally,
	\[
		\binom{K}{K/2}\;\ge\;\frac{2^{K}}{K+1},
	\]
	since $\binom{K}{K/2}$ is the largest of the $K+1$ binomial coefficients summing to~$2^{K}$. The instance has $K+1$ vertices, $K$ arcs and $K$ events, and the numbers $\delta_k$, $w^{E_k}$ and $\bar s_\rho(1)$ occurring in it are of magnitude at most $2^{K+1}$ and hence need $O(K)$ bits each, so the instance is encoded in $O(K^2)$ bits. The number of efficient temporal resilience profiles therefore grows superpolynomially in the encoding length, which establishes intractability.
\end{proof}

\section{Dichotomic Search for Bicriteria Linear Programs}\label{app:dichotomic-search}
In the following, we describe the dichotomic search used to compute the nondominated set of a
bicriteria linear program
\begin{equation}\label{eq:app-biclp}
    \min_{z \in Z}\; (F_1(z), F_2(z)) \coloneqq \bigl(c_1^\top z,\; c_2^\top z\bigr),
\end{equation}
where $Z \subseteq \mathbb{R}^n$ is a nonempty polyhedron and $c_1, c_2 \in
\mathbb{R}^n$. We assume that~\eqref{eq:app-biclp} is bounded, so that both
criteria attain a finite minimum over $X$. For a weight
$\lambda \in [0,1]$ the weighted sum problem can be written as 
\begin{equation}\label{eq:app-ws}
    \mathrm{WS}(\lambda):\qquad
    \min_{z \in Z}\; \lambda\, c_1^\top z + (1-\lambda)\, c_2^\top z ,
\end{equation}
for which denote by $z(\lambda)$ an optimal feasible solution of $\mathrm{WS}(\lambda)$
and by $F(\lambda) = (c_1^\top z(\lambda), c_2^\top z(\lambda))$ its image. For a bicriteria linear program~\eqref{eq:app-biclp}, it is well known that every nondominated image is the image of an optimal solution of $\mathrm{WS}(\lambda)$ for some $\lambda \in [0,1]$ and that the nondominated set is a connected, piecewise-linear curve consisting of finitely many line segments (cf.~\cite{matthias_ehrgott_multicriteria_2005}).
Consequently, it suffices to compute the extreme points of the polyhedron $Y  = F(Z) = \{F(z) : z \in Z \}$ that are nondominated. 

The dichotomic search does so. 
The search is initialized with the two lexicographic optima. Let $z^1$ minimize
$c_1^\top z$ over $Z$ and, among all such minimizers, $c_2^\top z$. Symmetrically,
let $z_2$ minimize $c_2^\top z$ and then $c_1^\top z$. Their images $y_1 = F(z_1)$
and $y_2 = F(z_2)$ are the two extreme nondominated images. The procedure maintains a queue of pairs of nondominated images~$\{y^a,y^b\}$, $y^a = (y^a_1,y^a_2)$ and $y^b = (y^b_1, y^b_2)$, satisfying  $y^a_1 < z^b_1$ and $y^a_2 > y^b_2$
In each iteration, it pops a pair~$\{y^a,y^b\}$ from the queue and solves the weighted sum problem~$\mathrm{WS}(\lambda^{ab})$, where
\begin{equation}\label{eq:app-lambda}
    \lambda^{ab} \;=\; \frac{y^a_2 - y^b_2}{(y^b_1 - y^a_1) + (y^a_2 - y^b_2)}.
\end{equation}
This yields an image $y^c = y(\lambda^{ab})$. If $\lambda^{ab} \cdot y^c_1 + (1-\lambda^{ab}) \cdot     y^{c}_2 = \lambda^{ab} \cdot y^a_1 + (1-\lambda^{ab}) \cdot y^a_2$, the segment $\overline{y^a y^b}$ is a subset of facet of the nondominated set and the
    pair is fathomed. Otherwise, $y^c$ is a new supported nondominated image, and the two
    subpairs $(y^a, y^c)$ and $(y^c, y^b)$ are added to the queue.
The search terminates when no open pair remains.
\begin{proposition}[\cite{matthias_ehrgott_multicriteria_2005}]\label{prop:app-correct}
The dichotomic search returns a set of nondominated images that contain all nondominated extreme points of the polyhedron $Y  = F(Z) = \{F(z) : z \in Z \}$.
\end{proposition}

\section{Instance Generation}\label{app:instances}

We generate random static minimum cost flow instances using the NETGEN grid algorithm of~\cite{klingman_netgen_1974}. These static instances are transformed
into networks with flows over time by replicating the arc set over time with varying the horizon over $T\in\{100,150,200,250,300\}$.
Arc capacities and costs are obtained by perturbing the original capacities and costs with i.i.d.\ samples from a normal distribution, while transit times are sampled independently from the discrete uniform distribution on $\{1,\ldots,5\}$. Storage costs and capacities are sampled independently from $[1,10]$ and $[1,50]$, respectively. 
We then solve the nominal flows over time problem, lexicographically maximizing the flow value accumulated over the horizon and minimizing the cost of the flow. 
Its optimal solution removes unnecessary slack from the instance. Storage, supply, and demand capacities are adjusted to values close to those utilized in the nominal solution, with small random perturbations. This tightening yields highly utilized networks. Since this adjustment may render the original solution infeasible, we then re-solve the nominal flow over time problem on the tightened instance under the same lexicographic objective, and take its optimal solution as the reference flow over time~$x^{\textup{ref}} = (f^{\textup{ref}}, s^{\textup{ref}}, d^{\textup{ref}}, u^{\textup{ref}})$ for the effectiveness criterion below.

We sample the adverse events following~\cite{Eshghali2023}. For each instance, impact rates for arc capacity reductions, supply reductions, storage capacity reductions, and demand increments are generated as i.i.d.\ samples from the uniform distribution on $[0,1]$. An impact rate $r$ defines a reduction of the corresponding original capacity, $\Delta^E \bar f_a(\theta) = -\,r\,\bar f_a(\theta)$ (and analogously for supply and storage), and an increment of demand, $\Delta^E \bar d_v(\theta) = r\,\bar d_v(\theta)$. For each instance, $10$ adverse events are sampled. Each adverse event onsets at $\theta^E$ drawn uniformly from $\{10,\ldots,20\}$ and ceases at a time drawn uniformly from $\{20,\ldots,\lfloor T/3\rfloor\}$.  Each event is assigned a weight $w^E$, sampled i.i.d.\ from the uniform distribution on $[0,1]$ and normalized so that $\sum_{E \in \mathcal{E}} w^E = 1$.
For the computational study, we realize the effectiveness criterion by bounds on the nominal flow value accumulated over time, and its associated costs. More precisely, deviations from the reference flow~$x^{\textup{ref}}$ are measured in the total unmet demand and the cost-weighted deviation of flow and storage, i.e.
\begin{align*}
    &H^{\textup{effect}}(x,x^{\textup{ref}}) = \Bigg(\; 
     \sum_{v \in V} \sum_{\theta=0}^{T} \big( u_v(\theta) - u^{\textup{ref}}_v(\theta) \big),\\
    &\qquad \sum_{a \in A} \sum_{\theta=0}^{T} c_a(\theta)\, \big( f_a(\theta) - f^{\textup{ref}}_a(\theta) \big) \\
    &\qquad + \sum_{v \in V} \sum_{\theta=0}^{T} c_v(\theta)\, \big( \textup{ex}_v(\theta) - \textup{ex}^{\textup{ref}}_v(\theta) \big)
    \;\Bigg),
\end{align*}
where $c_a(\theta)$ and $c_v(\theta)$ denote the transit and storage costs. Effectiveness is then enforced through the constraint
$H^{\textup{effect}}(x,x^{\textup{ref}}) \le B^{\textup{effect}}$ with a small tolerance $B^{\textup{effect}} = (B^{\textup{effect},\textup{unmet}}, B^{\textup{effect},\textup{cost}}) = (10^{-3},10^{-3})$. 

The per-unit cost of an expansive preparedness decision is sampled as $0.85\,T$ times an independent draw from the uniform distribution on $[1,3]$ for transit and storage capacity and on $[2,5]$ for supply; the per-unit cost of a reinforcement is sampled as $0.75\,T$ times an independent draw from the uniform distribution on $[1,3]$ for all three resource types, and the per-unit-and-per-time-step cost of a supporting preparedness decision is sampled from the uniform distribution on $[1,3]$ for transit and storage capacity and on $[2,5]$ for supply. The expansive and reinforcement rates are drawn independently for each arc and vertex, whereas the supporting rates are common to all arcs and vertices of the respective resource type. Consequently, per unit and for a single time step, supporting capacity is the least costly, followed by reinforcement and then expansive capacity. Finally, the preparedness budget is set to $T$ times an independent sample from the uniform distribution on $[0.7,1.5]$.
Table~\ref{tab:instance-characteristics} provides an overview over the characteristics of the instances.

\medskip

\begin{table}[htbp]
    \footnotesize
    \setlength{\tabcolsep}{4pt}
    \resizebox{\textwidth}{!}{%
    \begin{tabular}{@{}l r l c c c c c c c@{}}
        \toprule
        & &  & \multicolumn{4}{c}{Parameter ranges ($\text{med}\pm\text{s.d.}$)}
          & \multicolumn{3}{c}{Model size ($\text{med}\pm\text{s.d.}$)}\\
        \cmidrule(lr){4-7}\cmidrule(lr){8-10}
        $\lvert V\rvert$ & $\lvert A\rvert$ & $T$ 
        & Flow cap. & Supply & Storage cap. & Event len.
        & Cont.\ vars & Bin.\ vars & Constr.\\
        & & & & & & &  ($\times 10^{3}$) & & ($\times 10^{3}$)\\
        9  & 24 & 100--300  & $7.3 \pm 0.3$ & $0.81 \pm 0.04$ & $3.51 \pm 0.04$ & $44.4 \pm 20.1$ & $219 \pm 85$  & $3694 \pm 1439$ & $166 \pm 64$ \\
        12 & 34 & 100--300  & $7.5 \pm 0.3$ & $0.61 \pm 0.03$ & $3.50 \pm 0.02$ & $40.5 \pm 16.8$ & $300 \pm 116$ & $3710 \pm 1443$ & $223 \pm 86$ \\
        15 & 44 & 100--300  & $7.5 \pm 0.3$ & $0.49 \pm 0.03$ & $3.50 \pm 0.02$ & $46.3 \pm 16.0$ & $380 \pm 147$ & $3698 \pm 1442$ & $279 \pm 108$\\
        16 & 48 & 100--300 & $7.6 \pm 0.3$ & $0.45 \pm 0.02$ & $3.50 \pm 0.02$ & $41.0 \pm 20.6$ & $410 \pm 158$ & $3710 \pm 1436$ & $300 \pm 115$\\
        20 & 62 & 100--300  & $7.4 \pm 0.2$ & $0.36 \pm 0.02$ & $3.50 \pm 0.02$ & $45.5 \pm 16.5$ & $517 \pm 201$ & $3686 \pm 1441$ & $376 \pm 146$\\
        25 & 80 & 100--300  & $7.5 \pm 0.2$ & $0.28 \pm 0.02$ & $3.50 \pm 0.02$ & $47.6 \pm 18.7$ & $658 \pm 255$ & $3700 \pm 1441$ & $474 \pm 183$\\
        \bottomrule
    \end{tabular}%
    }
    \caption{Characteristics of the test instances, aggregated by  vertex count $\lvert V\rvert$ within each instance class. Every instance has $\lvert\mathcal{E}\rvert=10$ adverse events. For the parameter and model-size columns, entries report the median and    standard deviation ($\text{med}\pm\text{s.d.}$) taken over the per-instance means of the respective group; ``event length'' denotes     the duration $\ell^E$ (in time steps) of an adverse event. Continuous variables and constraints are reported in thousands ($\times 10^{3}$).}
    \label{tab:instance-characteristics}
\end{table}

\end{document}